%% file: almost_spanning_k-cycle.tex
\documentclass[11pt]{article}

\include{custom_stuff}

\title{Local resilience of an almost spanning $k$-cycle in random graphs}
\author{
  Nemanja \v{S}kori\'{c}\thanks{Institute of Theoretical Computer Science, ETH Z\"{u}rich, 8092 Z\"{u}rich, Switzerland \newline Email: \{\texttt{nskoric}\textbar \texttt{steger}\textbar \texttt{mtrujic}\}\texttt{@inf.ethz.ch}}
  \and
  Angelika Steger\footnotemark[1]
  \and
  Milo\v{s} Truji\'{c}\footnotemark[1] \textsuperscript{,}\thanks{author was supported by grant no.\ 200021 169242 of the Swiss National Science Foundation.}
}
\date{}

\begin{document}
\maketitle

\begin{abstract}
  The famous P\'{o}sa-Seymour conjecture, confirmed in 1998 by Koml\'{o}s,
  S\'{a}rk\"{o}zy, and Szemer\'{e}di, states that for any $k \geq 2$, every
  graph on $n$ vertices with minimum degree $kn/(k + 1)$ contains the $k$-th
  power of a Hamilton cycle. We extend this result to a sparse random setting.

  We show that for every $k \geq 2$ there exists $C > 0$ such that if $p \geq
  C(\log n/n)^{1/k}$ then w.h.p.\ every subgraph of a random graph $\Gnp$ with
  minimum degree at least $(k/(k + 1) + o(1))np$, contains the $k$-th power of a
  cycle on at least $(1 - o(1))n$ vertices, improving upon the recent results of
  Noever and Steger for $k = 2$, as well as Allen et al.\ for $k \geq 3$.

  Our result is almost best possible in three ways: for $p \ll n^{-1/k}$ the
  random graph $\Gnp$ w.h.p.\ does not contain the $k$-th power of any long
  cycle; there exist subgraphs of $\Gnp$ with minimum degree $(k/(k + 1) +
  o(1))np$ and $\Omega(p^{-2})$ vertices not belonging to triangles; there exist
  subgraphs of $\Gnp$ with minimum degree $(k/(k + 1) - o(1))np$ which do not
  contain the $k$-th power of a cycle on $(1 - o(1))n$ vertices.
\end{abstract}

\input{introduction}

\input{preliminaries}

\input{sparse_regularity}

\input{main}

\input{acknowledgements}

\bibliographystyle{abbrv}
\bibliography{references}

\end{document}

%% file: custom_stuff.tex
\usepackage[utf8]{inputenc}
\usepackage[english]{babel}
\usepackage[a4paper, margin=2.5cm]{geometry}
\usepackage{amsmath, amsthm, amssymb, amsfonts}
\usepackage{mathtools}
\usepackage[shortlabels]{enumitem}
\usepackage[font={small, it}]{caption}
\usepackage{subcaption}
\usepackage{graphicx}
\usepackage{color}
\usepackage{xcolor}
\usepackage{thmtools}
\usepackage{hyperref}

\newcommand{\cG}{\ensuremath{\mathcal G}}

\newcommand{\cK}{\ensuremath{\mathcal K}}

\newcommand{\cP}{\ensuremath{\mathcal P}}

\newcommand{\eps}{\varepsilon}
\renewcommand{\phi}{\varphi}

\DeclareMathOperator*{\E}{\mathbb{E}}
\DeclareMathOperator*{\N}{\mathbb{N}}

\DeclareMathOperator*{\R}{\mathbb{R}}

\DeclarePairedDelimiter\ceil{\lceil}{\rceil}
\DeclarePairedDelimiter\floor{\lfloor}{\rfloor}

\newcommand{\Gnp}{G_{n, p}}
\newcommand{\GNp}{G_{N, p}}

\newcommand{\binoms}[2]{{\textstyle\binom{#1}{#2}}} 

\declaretheorem[parent=section]{theorem}
\declaretheorem[sibling=theorem]{lemma}
\declaretheorem[sibling=theorem]{proposition}
\declaretheorem[sibling=theorem]{claim}
\declaretheorem[sibling=theorem]{corollary}
\declaretheorem[sibling=theorem,style=definition]{definition}

\setlist{itemsep=0.1em, topsep=0.1em}

\makeatletter
\def\thm@space@setup{%
  \thm@preskip=0.5em \thm@postskip=0.2em
}
\makeatother

\hypersetup{
  colorlinks,
  linkcolor={red!60!black},
  citecolor={green!50!black},
  urlcolor={blue!80!black}
}

\definecolor{RoyalAzure}{rgb}{0.0, 0.22, 0.66}

%% file: introduction.tex
\section{Introduction}

For a cycle $C$ and an integer $k \in \N$, the \emph{$k$-th power of a cycle} (\emph{$k$-cycle} for short) is obtained by including an edge between all pairs of vertices whose distance on $C$ is at most $k$. A classic result of Dirac~\cite{dirac1952some} states that any graph $G$ on $n \geq 3$ vertices with minimum degree $\delta(G) \geq n/2$ contains a Hamilton cycle. The famous P\'{o}sa-Seymour conjecture~\cite{erdos1964problem, seymour1973problem} generalises this statement to higher powers of a Hamilton cycle: every graph $G$ on $n$ vertices with minimum degree $\delta(G) \geq kn/(k + 1)$, contains the $k$-th power of a Hamilton cycle. Only after the appearance of powerful tools such as Szemer\'{e}di's regularity lemma and the blow-up lemma, the conjecture was resolved by Koml\'{o}s, Sark\"{o}zy, and Szemer\'{e}di~\cite{komlos1998proof}, at least for large enough values of $n$.
\begin{theorem}[Koml\'{o}s, S\'{a}rk\"{o}zy, Szemer\'{e}di~\cite{komlos1998proof}]\label{thm:KSS}
  For any positive integer $k$, there exists a natural number $n_0$ such that if $G$ has order $n$ with $n \geq n_0$ and
  \[
    \delta(G) \geq \frac{k}{k + 1}n,
  \]
  then $G$ contains the $k$-th power of a Hamilton cycle.
\end{theorem}

For a monotone increasing graph property $\cP$ and a graph $G$ which has property $\cP$, the \emph{resilience of $G$ with respect to $\cP$} measures how much one must change $G$ in order to destroy this property. \emph{Global resilience} measures this in terms of the number of edges, which for properties like containment of spanning structures is not very meaningful, as these can be easily destroyed just by isolating a vertex. Here we need the notion of \emph{local resilience}, which we define next.
\begin{definition}[Local resilience]
  Let $\cP$ be a monotone increasing graph property. We define the {\em local resilience} of a graph $G$ with respect to the property $\cP$ to be
  \begin{align*}
    r(G, \cP) := \min \{r \colon \exists H \subseteq G & \text{ such that all } v \in V(G)\text{ satisfy }\\
    & d_H(v) \leq r \cdot d_G(v) \text{ and } G - H \text{ does not have } \cP\}.
  \end{align*}
\end{definition}

From the results of Dirac and Koml\'{o}s, Sark\"{o}zy, and Szemer\'{e}di it follows that the local resilience of a complete graph $K_n$ with respect to containing the $k$-th power of a Hamilton cycle is at least $1/(k + 1)$. By considering a complete $(k + 1)$-partite graph with $k$ parts of size $(n - 1)/(k + 1)$ and one part of size $(n + k)/(k + 1)$ one easily checks that $1/(k + 1)$ is actually optimal.

A natural generalisation of resilience results for the complete graph is to consider random graphs, cf.\
Sudakov and Vu~\cite{sudakov2008local}. For a positive integer $N$ and a function $0 \leq p := p(N) \leq 1$, let $\GNp$ denote a graph on the vertex set $[N] = \{1, \ldots, N\}$, where each pair of vertices forms an edge independently and randomly with probability $p$.

In this paper, we study the local resilience of the property `contains the $k$-th power of a Hamilton cycle' for sparse random graphs. For $k = 1$, the question corresponds to the local resilience of Hamiltonicity, which was shown to be $1/2 + o(1)$ by Lee and Sudakov~\cite{lee2012dirac}, provided that $p \gg \log{N}/N$. An even stronger---the so-called `hitting-time' statement---has been recently shown by Nenadov and the last two authors~\cite{nenadov2017local}, and Montgomery~\cite{montgomery2017hamiltonicity}, independently. However, for $k \geq 2$, already the appearance threshold, i.e.\ the smallest $p$ such that the random graph $\GNp$ w.h.p.\ contains the $k$-th power of a Hamilton cycle, is yet to be fully understood. It is an easy first moment exercise to see that for $p \ll N^{-1/k}$ one does not expect the $k$-th power of any long cycle to appear in $\GNp$. On the other hand, for $k \geq 3$, K\"{u}hn and Osthus~\cite{kuhn2012posa} observed that the result of Riordan~\cite{riordan2000spanning} gives the correct answer: for $p \gg N^{-1/k}$ the random graph $\GNp$ w.h.p.\ contains the $k$-th power of a Hamilton cycle. The case $k = 2$ remains elusive with the currently best result $p \gg n^{-1/2} \log^4 n$ due to Nenadov and the first author~\cite{nenadov2016powers} which is optimal up to the logarithmic factor.

For local resilience, the best previously known results are by Noever and Steger~\cite{noever2017local} and Allen et al.~\cite{allen2016bandwidth}, for the square of a cycle and higher powers, respectively. The former states that for any $\eps, \gamma > 0$, if $p \geq N^{-1/2 + \gamma}$ the local resilience of $\GNp$ with respect to having the square of a cycle on $(1 - \eps)N$ vertices is $1/3 + o(1)$, while the latter is a variant of the bandwidth theorem for sparse random graphs and implies that for any $k \geq 2$, if $p \gg N^{-1/(2k)}$ the local resilience of $\GNp$ with respect to containing a $k$-cycle on all but $O(p^{-2})$ vertices, is $1/(k + 1) + o(1)$.

We improve upon both of these results by showing the local resilience statement for the property of containing the $k$-th power of an almost spanning cycle in a random graph $\GNp$ for density $p$ that is only by a factor of $(\log{N})^{1/k}$ away from the existence threshold.
\begin{restatable}{theorem}{MainTheorem}\label{thm:main-theorem}
  Let $k \geq 2$ be an integer. For every $\eps, \alpha > 0$ there exist positive constants $C(\eps, \alpha, k)$ and $c(\eps, \alpha, k)$ such that if $p \geq C(\frac{\log{N}}{N})^{1/k}$, with probability at least $1 - e^{- c \cdot (N/\log N)^{3/k}}$, every subgraph of $\GNp$ with minimum degree at least $(\frac{k}{k + 1} + \alpha)Np$ contains the $k$-th power of a cycle on at least $(1 - \eps)N$ vertices.
\end{restatable}

Our result is optimal with respect to the constant $1/(k + 1) $ in the local resilience and almost optimal with respect to the density $p$, as one cannot expect a long $k$-cycle to appear in $\GNp$ for $p \ll N^{-1/k}$. Moreover, by removing all edges in the neighbourhood of a vertex $v$, we can make sure that $v$ is not contained in any triangles. Thus for any $p = o(1)$, we have that the resilience of `containing a spanning $k$-cycle' is $o(1)$. In fact, Huang, Lee, and Sudakov~\cite{huang2012bandwidth} showed that by respecting the local resilience condition one can ensure that as many as $\Omega(p^{-2})$ vertices are not contained in a triangle. Note that the `leftover' given by Theorem~\ref{thm:main-theorem} is $\eps N$, which for $k = 2$ and the conjectured appearance threshold $p \geq N^{-1/2}$ would correspond to the optimal $O(p^{-2})$. However, as our bound on $p$ is slightly larger, this leaves a small gap with respect to the result of Allen et al.~\cite{allen2016bandwidth}.

The paper is organised as follows. In Section~\ref{sec:preliminaries} we establish an important property concerning subgraph counting in random graphs. In Section~\ref{sec:sparse-regularity} we introduce the sparse regularity lemma and a few essential notions that are extensively used throughout the paper. Furthermore, we give a simple lemma (Lemma~\ref{lem:regularity-partitioning}) which allows us to refine an $(\eps, p)$-regular partition into one in which the number of partition classes depends on $N$. In Section~\ref{sec:main-proof} we give a series of technical results which lead to the proof of our main tool (Lemma~\ref{lem:clique-expansion-lemma}) regarding expansion of cliques, and subsequently show how to combine everything in order to give a proof of our main result.

%% file: preliminaries.tex
\section{Preliminaries}\label{sec:preliminaries}

For an integer $n$ we write $[n] = \{1, \ldots, n\}$. For $a, b, c, d \in \R$, we write $a = (b \pm c)d$ to denote that $a$ lies in the interval $((b - c)d, (b + c)d)$. Additionally, given $k \in \N$ we let $(a \pm b)^k \subseteq (c \pm d)$ stand for $(c - d) \leq (a - b)^k, (a + b)^k \leq (c + d)$. We omit floors and ceilings whenever they are not crucial.

We use standard graph theory notation following the one from~\cite{bondy2008graphtheory}. In particular, given a graph $G$, we denote by $V(G)$ and $E(G)$ the set of its vertices and edges respectively, and write $v(G) = |V(G)|$ and $e(G) = |E(G)|$. For two (not necessarily disjoint) sets $X, Y \subseteq V(G)$ we denote by $E(X, Y) := \{ \{x, y\} \in E(G) \mid x \in X, y \in Y\}$ the set of edges with one endpoint in $X$ and the other in $Y$, and let $e(X, Y) = |E(X, Y)|$. Furthermore, we denote by $N_G(X, Y)$ the common neighbourhood of the vertices from $X$ in $Y$, i.e.\ $N_G(X, Y) := \{y \in Y \mid \forall x \in X \colon \{x, y\} \in E(G)\}$. We abbreviate $N_G(\{x\}, Y)$ to $N_G(x, Y)$ and write $N_G(X)$ for $N_G(X, V(G))$. We omit the subscript $G$ whenever it is clear from the context. We define the density of a pair $(X, Y)$ to be $d(X, Y) := e(X, Y)/(|X||Y|)$. For $t \geq 2$ pairwise disjoint subsets of vertices $V_1, \ldots, V_t \subseteq V(G)$, we let $G[V_1, \ldots, V_t]$ denote the $t$-partite graph induced by those subsets.

For an integer $\ell \geq 2$ we say that a path $P_\ell$ has \emph{length} $\ell$ if $P_\ell$ consists of $\ell$ vertices. We define the $k$-th power of a path, for short a $k$-\emph{path}, as the graph obtained by adding an edge between any two vertices at distance at most $k$ on the path. We define a $k$-\emph{cycle} analogously and write $P^{k}$ and $C^{k}$ for a $k$-path and $k$-cycle, respectively.

We also use the standard asymptotic notation $o, O, \omega$, and $\Omega$. Throughout the paper $\log$ denotes the natural logarithm. Finally, we often use subscripts with constants such as $C_{3.5}$ to indicate that $C_{3.5}$ is a constant given by Claim/Lemma/Proposition/Theorem 3.5.

\subsection{Subgraph counting in random graphs}
Let $H$ be a graph on the vertex set $\{v_1, \ldots, v_t\}$. We write $G(H, n, p)$ for a random graph on the vertex set $V = V_1 \cup \ldots \cup V_t$ where $|V_i| = n$, for all $i \in [t]$, and $\Pr[\{u, v\} \in E(G)] = p$ if $v \in V_i$ and $u \in V_j$ for some $\{v_i, v_j\} \in E(H)$. Let $X_H(n,p)$ be a random variable denoting the number of copies of $H$ in $G(H, n, p)$. For $H = K_t$, DeMarco and Kahn~\cite{demarco2012tight} showed exponential tail bounds on $X_H(n, p)$.
\begin{theorem}[Theorem 2.3 in~\cite{demarco2012tight}]\label{thm:DMK}
  Let $t \geq 2$ be an integer, $\delta$ a positive real, and $p \geq n^{-2/(t - 1)}$. Then there exists a constant $c(\delta, t) > 0$, such that
  \[
    \Pr \big[ X_{K_t}(n, p) \geq (1 + \delta)\E[X_{K_t}(n, p)] \big] \leq e^{-c \cdot \min\{n^2p^{t - 1}\log{(1/p)}, n^{t}p^{\binom{t}{2}}\}}.
  \]
\end{theorem}
An easy consequence of Theorem~\ref{thm:DMK} and a union bound is that with high probability all suitably large $t$-partite subgraphs of a random graph $\GNp$ contain not many more than the expected number of copies of $K_t$.
\begin{proposition}\label{prop:counting-cliques}
  Let $t \geq 2$ be an integer. Then for all $\eps > 0$, there exist positive constants $C(\eps, t)$ and $c(\eps, t)$, such that for $p \in (0, 1)$ and every integer $n \ge C p^{-t + 1} \log N$ a random graph $G \sim \GNp$ has the following property with probability at least $1- e^{-c \cdot n^2 p^{t-1}}$. Suppose $S_1, \ldots, S_t \subseteq V(G)$ are pairwise disjoint sets of size $|S_i| \geq n$, then $G[S_1, \ldots, S_t]$ contains at most $(1 + \eps) (\prod_{i = 1}^{t} |S_i|) p^{e(K_t)}$ copies of $K_t$.
\end{proposition}
\begin{proof}
  Let $\delta > 0$ satisfy $(1 + \delta)^{t + 1} = 1 + \eps$ and take $\tilde c = c_{\ref{thm:DMK}}(\delta, t)$, $c = \tilde c \delta^2/3$, and $C = 3t/(\delta \tilde c)$. Set $\tilde n = \ceil{\delta n}$. Consider $t$ pairwise disjoint sets each of size $\tilde n$. By Theorem~\ref{thm:DMK} the probability that $\GNp$ is such that these sets induce more than $(1 + \delta) \tilde n^t p^{e(K_t)}$ copies of $K_t$ is bounded by $e^{-\tilde c \cdot \tilde n^2 p^{t - 1}}$, where we used that the definition of $\tilde n$ implies $\min\{ \tilde n^2 p^{t - 1} \log(1/p), \tilde n^{t} p^{\binom{t}{2}} \} \geq \tilde n^2 p^{t - 1}$ (for $t = 2$ this term is sharp). Indeed, we may apply Theorem~\ref{thm:DMK} since by the assumption on $n$ we have
  \[
    p \geq \Big( \frac{C\log N}{n} \Big)^{1/(t - 1)} \geq \Big( \frac{C \delta \log N}{\tilde n} \Big)^{1/(t - 1)} \geq \tilde n^{-2/(t - 1)},
  \]
  with room to spare.
  Applying a union bound argument over all the choices for such sets, we thus see that the probability of $\GNp$ containing any such family of sets is at most
  \[
    {\binom{N}{\tilde n}}^{t} e^{-\tilde c \cdot \tilde n^2 p^{t - 1}} \leq e^{t \tilde n\log{N}} \cdot e^{-\tilde c \cdot \tilde n^2 p^{t - 1}} \leq e^{- (\tilde c/2) \tilde n^2 p^{t-1}} \le e^{-c n^2 p^{t-1}},
  \]
  where the last two inequalities follow from the definition of $\tilde n$, $C$, and $c$. In the following we condition on the event that $\GNp$ does not contain such a family of sets.

  Let $S_1, \ldots, S_t \subseteq V(G)$ be subsets which satisfy the assumptions of the proposition. Note that by our choice of $\tilde n$ we have $|S_i| \geq \delta^{-1} \tilde n$ for all $i \in [t]$. Let $\tilde S_1, \ldots, \tilde S_t$ be sets obtained by adding vertices to $S_1, \ldots, S_t$ until $|\tilde S_i|$ is divisible by $\tilde n$, for every $i \in [t]$. Observe that $|\tilde S_i| < |S_i| + \tilde n \le (1 + \delta) |S_i|$. We partition every $\tilde S_i$ arbitrarily into $k_i := |\tilde S_i| / \tilde n $ sets of size $\tilde n$. As we conditioned on the fact that the number of copies of $K_t$ in sets of size $\tilde n$ does not exceed the expectation by more than a factor of $\delta$, we deduce that the number of such copies induced by $S_1, \ldots, S_t$ is at most
  \begin{align*}
    \bigg( \prod_{i = 1}^t k_i \bigg) \cdot (1 + \delta) \tilde n^t p^{e(K_t)} &\le (1 + \delta) \bigg( \prod_{i = 1}^{t} k_i \tilde n \bigg) \cdot p^{e(K_t)} \le (1 + \delta) \bigg( \prod_{i = 1}^{t} (|S_i| + \tilde n) \bigg) \cdot p^{e(K_t)} \\
    &\le (1 + \delta)^{t + 1} \prod_{i = 1}^{t} |S_i| \cdot p^{e(K_t)} \le (1 + \eps) \prod_{i = 1}^{t} |S_i| \cdot p^{e(K_t)},
  \end{align*}
  as claimed.
\end{proof}

%% file: sparse_regularity.tex
\section{Sparse regularity lemma}\label{sec:sparse-regularity}

For disjoint subsets $V_1, V_2 \subseteq V(G)$, the pair $(V_1, V_2)$ is said to be $(\eps, p)$-\emph{regular} if $|d(V_1', V_2') - d(V_1, V_2)| \leq \eps p$, for all $V_1' \subseteq V_1$ and $V_2' \subseteq V_2$ such that $|V_1'| \geq \eps|V_1|$ and $|V_2'| \geq \eps|V_2|$. A $t$-tuple $(V_1, \ldots, V_t)$ is said to be $(\eps, p)$-\emph{regular} if $(V_i, V_j)$ forms an $(\eps, p)$-regular pair for all $1 \leq i < j \leq t$. The following simple hereditary property of regular pairs follows directly from the definition.
\begin{proposition}\label{prop:large-subsets-inheritance}
  Let $0 < \eps_1 < \eps_2 \leq 1/2$ be constants. Let $(V_1, V_2)$ be an $(\eps_1, p)$-regular pair for some $0 < p < 1$ and let $V_i' \subseteq V_i$ be arbitrary subsets such that $|V_i'| \geq \eps_2|V_i|$. Then $(V_1', V_2')$ is an $(\eps_1/\eps_2, p)$-regular pair of density $d(V_1, V_2) \pm \eps_1p$.
\end{proposition}

A partition $(V_i)_{i = 0}^{k}$ of the vertex set $V(G)$ is called an $(\eps, p)$-\emph{regular partition} if: the exceptional set $V_0$ is of size at most $\eps n$, sets $V_1, \ldots, V_k$ have equal sizes, and all but at most $\eps k^2$ many pairs $(V_i, V_j)$ are $(\eps, p)$-regular. The original sparse regularity lemma due to Kohayakawa and R\"{o}dl~\cite{kohayakawa1997szemeredi, kohayakawa2003szemeredi}, required the graph to fulfil a certain density condition. Here we state a variant recently given by Scott~\cite{scott2011szemeredi}, where such a condition is not necessary.
\begin{theorem}[Sparse regularity lemma,~\cite{scott2011szemeredi}]\label{thm:SRL}
  For any $\eps > 0$ and $m \geq 1$, there exists a constant $M(\eps, m) \geq m$ such that for any $p \in [0, 1]$, any graph $G$ on at least $M$ vertices of density $p$ admits an $(\eps, p)$-regular partition $(V_i)_{i = 0}^{k}$ with exceptional class $V_0$ such that $m \leq k \leq M$.
\end{theorem}

Our proof strategy requires the graph to have a particularly `nice' regular partition, namely one in which we can control the density between regular pairs. Such a statement follows from Theorem~\ref{thm:SRL} by standard arguments; a proof can be found e.g.~in \cite{noever2017local}.
\begin{corollary}[Corollary~2.6 in~\cite{noever2017local}]\label{cor:nice-regularity}
  For every $\mu, \nu, \eps > 0$ and $m \geq 1$, there exist $d(\nu)$ and $M(\eps, m) \geq m$ such that for $p = \omega(1/N)$, with probability at least $1 - e^{-\Omega(N^2p)}$ the following holds. Every spanning subgraph $G \subseteq \GNp$ with minimum degree $\delta(G) \geq (\mu + \nu) Np$ contains a partition of the vertices $V(G) = V_0 \cup V_1 \cup \ldots \cup V_k$, where $m \leq k \leq M$, such that $|V_0| \leq \eps N$, $|V_1| = \ldots = |V_k|$, and such that for every $i$ there exist at least $\mu k$ indices $j \in [k] \setminus \{i\}$ with $G[V_i, V_j]$ being an $(\eps, p)$-regular pair of density $dp$.
\end{corollary}

\subsection{Regularity inheritance}
The key ingredients of our proof strategy---namely, Lemma~\ref{lem:regularity-partitioning} and Lemma~\ref{lem:clique-expansion-lemma}---require that besides large subsets (see, Proposition~\ref{prop:large-subsets-inheritance}), \emph{most} of the \emph{small} subsets inherit $(\eps, p)$-regularity. In particular, a random pair of subsets of size roughly $p^{-1}$ is $(\eps', p)$-regular with high probability, for a slightly worse $\eps'$. A somewhat weaker property which serves as the main step towards achieving that goal was first observed by Gerke, Kohayakawa, R\"{o}dl, and Steger~\cite{gerke2007small}.
\begin{theorem}[Corollary 3.9 in \cite{gerke2007small}]\label{thm:GKRS}
  For all $0 < \beta, \eps' < 1$, there exist $\eps_0(\beta, \eps')$ and $C(\eps')$ such that, for any $0 < \eps \leq \eps_0$ and $0 < p < 1$, the following holds. Suppose $G = (V_1 \cup V_2, E)$ is an $(\eps, p)$-regular graph of density $dp$ and suppose $q_1, q_2 \geq C(dp)^{-1}$. Let $N$ be the number of pairs $(Q_1, Q_2)$ with $Q_i \subseteq V_i$ and $|Q_i| = q_i$ ($i = 1, 2$), and such that there are $\tilde{Q}_i \subseteq Q_i$ with $|\tilde{Q}_i| \geq (1 - \eps')|Q_i|$ for which we have
  \begin{enumerate}[label=(\roman*)]
    \item $G'= G[\tilde{Q}_1, \tilde{Q}_2]$ is $(\eps', p)$-regular,
    \item the density $d'p$ of $G'$ satisfies $(1 - \eps')dp \leq d'p \leq (1 + \eps')dp$.
  \end{enumerate}
  Then
  \[
    N \geq (1 - \beta^{\min\{q_1, q_2\}})\binom{|V_1|}{q_1}\binom{|V_2|}{q_2}.
  \]
\end{theorem}

It turns out that when we are dealing with a subset of a random graph, we do not need to look into subsets of sets $Q_1$ and $Q_2$, but the sets themselves span a regular pair. We point out that such a statement is not true in general (see, Section 3.3 in \cite{gerke2007small}).
\begin{corollary}\label{cor:small-subsets-inheritance}
  For all $0 < \beta, \eps', d, p < 1$, there exist positive constants $\eps_0(\beta, \eps', d)$, $D(\eps')$, and $c(\eps')$ such that, for every $q_0 \ge D p^{-1} \log N$ and $0 < \eps \le \eps_0$, a random graph $\GNp$ has the following property with probability at least $1 - e^{-c \cdot q_0^2 p}$. Suppose $G = (V_1 \cup V_2, E)$ is an $(\eps, p)$-regular graph of density $dp$ which is a subgraph of $\GNp$. Then for all $q_1, q_2 \ge q_0$, there are at least
  \[
    (1 - \beta^{\min\{q_1, q_2\}})\binom{|V_1|}{q_1}\binom{|V_2|}{q_2}
  \]
  sets $Q_i \subseteq V_i$ of size $|Q_i| = q_i$ ($i = 1, 2$) which induce an $(\eps', p)$-regular graph of density $(1 \pm \eps')dp$.
\end{corollary}
\begin{proof}
  The proof is straightforward: we apply Theorem~\ref{thm:GKRS} for an $\eps''$ that we choose small enough (depending on $\eps'$ and $d$) and then use that the number of edges in $\GNp$ between any sufficiently large sets is sharply concentrated around the expectation. This allows us to show that all pairs $(Q_1, Q_2)$ for which subsets $\tilde Q_i$ exist and satisfy the condition of Theorem~\ref{thm:GKRS} are in fact $(\eps', p)$-regular themselves. We now make this more precise.

  For given $\beta, \eps', d$, let $\eps''$ be small enough depending on $\eps'$ and $d$ such that the calculations below hold, $\eps_0 = \eps_{0_{\ref{thm:GKRS}}}(\beta, \eps'')$, $D = \max\{C_{\ref{prop:counting-cliques}}(\eps'', 2)/\eps'', C_{\ref{thm:GKRS}}(\eps'')\}$, and $c = c_{\ref{prop:counting-cliques}}(\eps'', 2)$. Let $Q_1, Q_2, \tilde Q_1, \tilde Q_2$ be as in Theorem \ref{thm:GKRS} applied with $\beta$ and $\eps''$ as $\eps'$. Consider arbitrary subsets $Q_i' \subseteq Q_i$ of size at least $\eps' |Q_i|$, and set $\tilde Q_i' := \tilde Q_i \cap Q_i'$. Observe that
  \[
    |\tilde Q_i'| \geq (\eps' - \eps'')|{Q}_i| \ge (\eps'-\eps'') |\tilde{Q}_i| \ge \eps''|\tilde{Q}_i|,
  \]
  as $\eps''$ is sufficiently small compared to $\eps'$. We thus know that
  \[
    d(\tilde Q_1', \tilde Q_2') = d(\tilde Q_1, \tilde Q_2) \pm \eps'' p = (1\pm \eps'')dp \pm \eps'' p = (1 \pm (\eps'' + \eps''/d)) dp.
  \]
  What remains is to show that $d(Q_1', Q_2')$ is sufficiently close to $d(\tilde Q_1', \tilde Q_2')$, which in particular implies that $d(Q_1', Q_2')$ is close to $d(Q_1, Q_2)$.

  By applying Proposition~\ref{prop:counting-cliques} with $t = 2$ and $\eps''$ as $\eps$, we get that with probability at least $1 - e^{-c \cdot q_0^2p}$ any two disjoint sets $S_1, S_2$ of size at least $\eps'' q_0$ satisfy $e(S_1, S_2) \le (1 + \eps'') |S_1| |S_2| p$. Since $|Q'_i| \ge \eps' |Q_i| \ge \eps'' q_0$ and
  \[
    |Q_i' \setminus \tilde Q_i'| \leq \eps'' |Q_i| \leq (\eps''/\eps')|Q_i'|,
  \]
  it holds that
  \[
    d(Q'_2 \setminus \tilde Q'_2, Q'_1), d(Q'_1 \setminus \tilde Q'_1, Q'_2) \le (1 + \eps'') (\eps'' / \eps') |Q'_1| |Q'_2|p.
  \]
  Note that if $Q_i' \setminus \tilde Q_i'$ are not of size at least $\eps''|Q_i|$ we may simply take arbitrary (super)sets of size $\eps''|Q_i|$. From here we get
  \begin{align*}
    e(Q_1', Q_2') &\leq e(\tilde Q_1', \tilde Q_2') + e(Q_1' \setminus \tilde Q_1', Q_2') + e(Q_1', Q_2' \setminus \tilde Q_2') \\
    &\leq e(\tilde Q_1', \tilde Q_2') + 2 (\eps''/\eps') (1 + \eps'') |Q_1'||Q_2'|p.
  \end{align*}
  Using the fact that $(1 - \eps''/\eps') \leq |\tilde Q_i'|/|Q_i'| \leq 1$ and the previous inequality, we obtain
  \[
    (1 - \eps''/\eps')^2 d(\tilde Q_1', \tilde Q_2') \le d(Q_1', Q_2') \le d(\tilde Q_1',\tilde Q_2') + \frac{2\eps'' (1 + \eps'')}{\eps'}p.
  \]
  As $d(\tilde Q_1', \tilde Q_2') = (1 \pm (\eps'' + \eps''/d)) dp$, we see that whenever $\eps'' = \eps''(\eps', d)$ is small enough then $d(Q_1', Q_2') = dp \pm \eps'p = d(Q_1, Q_2) \pm\eps'p$, as claimed.
\end{proof}

The next lemma shows that in a random graph any $(\eps, p)$-regular partition with $k$ classes can be used to generate $(\eps', p)$-regular partitions with the number of classes growing with $N$, for a slightly worse $\eps'$. We remark that such a lemma may be of independent interest and prove useful in other problems concerning embeddings of large structures into random graphs.
\begin{lemma}\label{lem:regularity-partitioning}
  For all $0 < \eps', d, p < 1$ there exist positive constants $\eps_0(\eps', d)$, $D(\eps')$, and $c(\eps')$, such that for every integer $q \ge D p^{-1} \log N$ and every $0 < \eps \leq \eps_0$, the following holds with probability at least $1 - e^{ - c \cdot q^2 p}$. Suppose $(V_i)_{i = 0}^{k}$ is an $(\eps, p)$-regular partition of the vertex set of a spanning subgraph $G$ of a random graph $\GNp$. Then there exists a partition $(V_i^j)_{j = 0}^{t}$ of each $V_i$ such that:
  \begin{enumerate}[label=(\roman*)]
    \item\label{reg-size} $|V_{i}^{0}| \leq q$ and $|V_i^1| = \ldots = |V_i^t| = q$,
    \item\label{reg-inheritance} if $(V_{i_1}, V_{i_2})$ is an $(\eps, p)$-regular pair of density $dp$, then $(V_{i_1}^{j_1}, V_{i_2}^{j_2})$ is an $(\eps', p)$-regular pair of density $(1 \pm \eps')dp$, for all $j_1, j_2 \in [t]$.
  \end{enumerate}
\end{lemma}
\begin{proof}
  Let $\beta = 1/2$, $\eps_0 = \eps_{0_{\ref{cor:small-subsets-inheritance}}}(\beta, \eps', d)$, $D = \max \{D_{\ref{cor:small-subsets-inheritance}}(\eps'), 10\}$, $c' = c_{\ref{cor:small-subsets-inheritance}}(\eps')$, $c = c' / 2$, and suppose that $(V_i)_{i = 0}^{k}$ is an $(\eps, p)$-regular partition of $V(G)$. Let $q \geq Dp^{-1} \log{N}$ be fixed. From here on, we assume that the conclusions of Corollary \ref{cor:small-subsets-inheritance} hold when applied with pair $(V_{i_1}, V_{i_2})$, for every two $i_1, i_2 \in [t]$. This is true with probability at least $1 - t^2 e^{-c' \cdot q^2p} \ge 1 - e^{-c \cdot q^2 p}$.

  For every $i \in [k]$, take $(V_{i}^{j})_{j = 0}^{t}$ to be a partition chosen u.a.r.\ among all partitions satisfying $|V_{i}^{1}| = \ldots = |V_{i}^{t}| = q$, with a leftover set $|V_{i}^{0}| \leq q$. We show that such a partition w.h.p.\ satisfies the required properties.

  By Corollary~\ref{cor:small-subsets-inheritance} we know that for every $(\eps, p)$-regular pair $(V_{i_1}, V_{i_2})$ of density $dp$ there are at most
  \[
    \beta^{q} \binom{|V_{i_1}|}{q} \binom{|V_{i_2}|}{q}
  \]
  sets $Q_{i_1} \subseteq V_{i_1}, Q_{i_2} \subseteq V_{i_2}$, of size $q$ which do not induce an $(\eps', p)$-regular pair of density $(1 \pm \eps')dp$. Therefore, the probability that a pair $(V_{i_1}^{j_1}, V_{i_2}^{j_2})$ does not satisfy $\ref{reg-inheritance}$ is at most $\beta^{q}$, for any fixed $j_1, j_2 \in [t]$.

  Applying a union bound over at most $\binoms{k}{2}$ possible $(\eps, p)$-regular pairs $(V_{i_1}, V_{i_2})$ and all $j_1, j_2 \in [t]$, we obtain that the probability that our partition $(V_{i}^{j})_{j = 1}^{t}$ does not satisfy $\ref{reg-inheritance}$, is at most
  \[
     \binom{k}{2} t^2 \beta^{q} \leq (kt)^2 \cdot \beta^{q} \le N^2 \cdot 2^{-\frac{10 \log N}{p}} = o(1),
  \]
  where the last inequality holds as $\beta = 1/2$ and $q \ge 10 p^{-1} \log N$. This, in particular, implies that there {\em exists} a partition of $V(G)$ as required, satisfying both $\ref{reg-size}$ and $\ref{reg-inheritance}$, which completes the proof.
\end{proof}

\subsection{Typical vertices and blow-ups}
One nice property of regularity is that in an $(\eps, p)$-regular $t$-tuple most of the vertices have neighbourhoods in the other sets of roughly the expected size. For us it is useful if these neighbourhoods also induce regular pairs. We capture this in the following definition.
\begin{definition}[$\eps$-\textbf{typical}]\label{def:typical-tuples}
  Let $t \geq 3$ be an integer and $(V_1, \ldots, V_t)$ a $t$-tuple with densities $d_{ij}p$ between $V_i$ and $V_j$. A vertex $v \in V_i$, for $i \in [t]$, is said to be $\eps$-\emph{typical} if:
  \begin{enumerate}[(i), font=\itshape]
    \item $|N_j| = (1 \pm \eps)d_{ij}|V_j|p$ for all $j \in [t] \setminus \{i\}$, where $N_j := N(v, V_j)$, and
    \item\label{typical-tuple-reg-inheritance} $(N_{j}, N_{k})$ is an $(\eps, p)$-regular pair of density $(1 \pm \eps)d_{jk}p$, for all $j, k \in [t] \setminus \{i\}$, $j \neq k$.
  \end{enumerate}
  Furthermore, we say that the $t$-tuple $(V_1, \ldots, V_t)$ is $\eps$-\emph{typical} if it is $(\eps, p)$-regular and for each $i \in [t]$ all but at most $\eps|V_i|$ vertices are $\eps$-\emph{typical}.
\end{definition}

Given a graph $H$ on the vertex set $\{1, \ldots, t\}$ and sequences $\mathbf n = (n_i)_{i \in V(H)}$ of positive integers and $\mathbf d = (d_{ij})_{\{i, j\} \in E(H)}$ of positive reals, we denote by $\cG(H, \mathbf n, \mathbf d, \eps, p)$ the class of graphs that consist of $|V(H)|$ disjoint sets of size $n_i$, each representing a vertex of $H$, and an $(\eps, p)$-regular graph of density $d_{ij}p$ between two sets whenever the corresponding vertices are adjacent in $H$. Similarly, for positive reals $\eta, \alpha$, we denote by $\cG(H, n, \eta, \alpha, \eps, p)$ the class of graphs as above where $n_i \in [\eta n, n]$ for all $i \in V(H)$ and $d_{ij} \in (1 \pm \eps)\alpha$ for all $\{i, j\} \in E(H)$.

It turns out that almost all graphs in $\cG(K_t, \mathbf n, \mathbf d, \eps, p)$ are $\eps'$-typical, whenever $\eps$ is sufficiently small compared to $\eps'$ and the class sizes are not too small. The following lemma makes this precise.
\begin{lemma}\label{lem:bad-graphs-count}
  Let $t \geq 3$ be an integer. For all $\alpha, \beta, \eps' > 0$ there exist positive constants $\eps_0(\alpha, \beta, \eps')$ and $D(\alpha, \eps')$ such that for every $\eps \leq \eps_0$, $\mathbf n$ with $n_i \geq n$, $\mathbf d$ with $d_{ij} \geq \alpha$, and $p \geq Dn^{-1/2}$, all but at most
  \[
    \beta^{\alpha n^2p} \prod_{1 \leq i < j \leq t} \binom{n_i n_j}{d_{ij} n_i n_j p}
  \]
  graphs in $\cG(K_t, \mathbf n, \mathbf d, \eps, p)$ are $\eps'$-typical, provided that $n$ is sufficiently large.
\end{lemma}

Lemma~\ref{lem:bad-graphs-count} is proven by a straightforward modification of the proof of \cite[Lemma~5.1]{gerke2005sparse}, we omit the details. Moreover, the number of non-typical graphs is so small that a simple union bound implies that w.h.p.\ in a random graph $\GNp$ every large enough regular tuple is also typical.
\begin{corollary}\label{cor:regular-to-typical}
  Let $t \geq 3$ be an integer. For all $\alpha, \eps', \eta > 0$, there exist positive constants $\eps_0(\alpha, \eps', \eta, t)$, $C(\alpha, \eps', \eta, t)$, and $c(\alpha, \eps', \eta, t)$ such that for every integer $n \geq \max\{Cp^{-2}, Cp^{-1}\log{N}\}$ a random graph $\GNp$ satisfies the following with probability at least $1 - e^{-c \cdot n^2 p}$. Every subgraph $G$ of $\GNp$ in $\cG(K_{t}, n, \eta, \alpha, \eps, p)$ is $\eps'$-typical, provided that $\eps \leq \eps_0$.
\end{corollary}
\begin{proof}
  For given $\alpha, \eps', t$, let $\tilde\alpha = (1 - \eps')\alpha \eta^2/t^2$, $\tilde \beta = \tilde \alpha/e^2$, and choose $\beta$ such that equation \eqref{eq:bad-graphs-bound} below holds. Lastly, take $C = \max\{ D_{\ref{lem:bad-graphs-count}}((1 - \eps')\alpha, \eps')/\eta, 2t / (\eta \tilde \alpha )\}$, $\eps_0 = \min\{\eps', \eps_{0_{\ref{lem:bad-graphs-count}}}((1 - \eps')\alpha, \beta, \eps')\}$, and $c = \tilde \alpha \eta^2 /4$.

  Fix sequences $\mathbf n = (n_i)_{i \in V(K_t)}$ with all $n_i \in [\eta n, n]$ and $\mathbf d = (d_{ij})_{\{i, j\} \in E(K_t)}$ with all $d_{ij} \in (1 \pm \eps)\alpha$. Assume $G$ is a graph which belongs to $\cG(K_t, \mathbf n, \mathbf d, \eps, p)$ but is not $\eps'$-typical. Then, making use of Vandermonde's identity (cf.\ e.g.~\cite{aigner2007course}), i.e.\ the fact that $\binom{x + y}{k} = \sum_{i} \binom{x}{i} \binom{y}{k - i}$, we conclude that $G$ must be one of at most
  \begin{equation}\label{eq:bad-graphs-bound}
    \beta^{(1 - \eps')\alpha \eta^2 n^2 p} \prod_{1 \leq i < j \leq t} \binom{n_in_j}{d_{ij} n_in_j p}\leq \beta^{(1 - \eps')\alpha \eta^2 n^2 p} \binom{\sum_{1 \leq i < j \leq t} n_in_j}{\sum_{1 \leq i < j \leq t} d_{ij} n_in_j p} \leq \tilde \beta^{\tilde m} \binom{\tilde n^2}{\tilde m}
  \end{equation}
  graphs enumerated by Lemma~\ref{lem:bad-graphs-count}, where $\tilde n = \sum_{1 \leq i \leq t} n_i$ and $\tilde m = \sum_{1 \leq i < j \leq t} d_{ij} n_i n_j p$. One easily sees that our choice of $\tilde \alpha$ implies $\tilde m \geq \tilde\alpha \tilde n^2 p$. Hence, the expected number of copies of $G$ in $\GNp$ is at most
  \begin{align*}
    \bigg( \prod_{i=1}^t \binom{N}{n_i} \bigg) \cdot \beta^{\tilde m} \binom{\tilde n^2}{\tilde m} p^{\tilde m} & \leq N^{\tilde n} \beta^{\tilde m} \Big( \frac{e}{\tilde\alpha p} \Big)^{\tilde m} p^{\tilde m} \leq e^{\tilde n \log N} \Big( \frac{\tilde\alpha}{e^2} \Big)^{\tilde m} \Big( \frac{e}{\tilde\alpha} \Big)^{\tilde m} \\
    & \leq e^{\tilde n \log N - \tilde\alpha \tilde n^2p} \le e^{-(\tilde \alpha /2) \tilde n^2 p},
  \end{align*}
  where the last inequality follows from the choice of $C$. Applying a union bound for all sequences $\mathbf n$ and $\mathbf d$ we conclude that $\GNp$ with probability at least $1 - e^{-c \cdot n^2 p}$ does not contain any such graph as a subgraph.
\end{proof}

Let $H$ be a fixed graph on the vertex set $\{1, \ldots, t\}$ and let $G = (V_1 , \ldots, V_t)$ be a $t$-partite graph. We say that a $t$-tuple $(v_1, \ldots, v_t)$ is a \emph{canonical copy} of $H$ in $G$ if $v_i \in V_i$ for every $i \in V(H)$ and $\{v_i, v_j\} \in E(G)$ whenever $\{i, j\} \in E(H)$. We denote by $H(V_1, \ldots, V_t)$ the set of all canonical copies of $H$ in $(V_1, \ldots, V_t)$. Note that for $t = 2$ the number of canonical copies of $K_2$ in $(V_1,V_2)$ is exactly equal to the number of edges between $V_1$ and $V_2$, which is $|V_1||V_2|$ times the density of $(V_1,V_2)$. For $t \ge 3$ we get a similar statement whenever the corresponding $t$-tuple is sufficiently regular. The following lemma makes this precise.
\begin{lemma}\label{lem:counting-lemma}
  Let $t \geq 2$ be an integer. Then for all $\alpha, \delta, \eta > 0$, there exist positive constants $\eps_0(\alpha, \delta, \eta, t)$, $C(\alpha, \delta, \eta, t)$, and $c(\alpha, \delta, \eta, t)$ such that for every integer $n \geq \max\{ Cp^{-t}, Cp^{-t + 1}\log{N} \}$ a random graph $\GNp$ satisfies the following with probability at least $1 - e^{-c \cdot n^2 p^{2(t-2) + 1}}$. Every subgraph $G$ of $\GNp$ in $\cG(K_{t}, n, \eta, \alpha, \eps, p)$ contains
  \[
    (1 \pm \delta) \bigg(\prod_{1 \le i \le t} n_i \bigg) \bigg(\prod_{1 \le i < j \le t} d_{ij} p\bigg)
  \]
  canonical copies of $K_t$, provided that $\eps \leq \eps_0$.
\end{lemma}
\begin{proof}
  The proof is by induction on $t$. For $t = 2$ there is nothing to show by the discussion above.
  Consider some $t \geq 3$.

  For given $\delta, \alpha, \eta$, let us choose $\delta'$ and $\tilde\eps$ such that $(1 \pm \delta')^{2} \subseteq (1 \pm \delta)$ and $\big( (1 \pm \tilde \eps)^{t^{2}} \pm 2 \tilde \eps/\alpha^{t^{2}} \big) \subseteq (1 \pm \delta')$. Furthermore, we define $\alpha_{t - 1} = \alpha$, $\delta_{t - 1} = \delta'$, $\eta_{t - 1} = (1 - \tilde\eps)^{2}\eta / (1 + \tilde\eps)^{2}$, $\eps_{t - 1} = \eps_{0_{\ref{lem:counting-lemma}}}(\alpha_{t - 1}, \delta_{t - 1}, \eta_{t - 1}, t-1)$, $C_{t - 1} = C_{\ref{lem:counting-lemma}}(\alpha_{t - 1}, \delta_{t - 1}, \eta_{t - 1}, t-1)$, $c_{t - 1} = c_{\ref{lem:counting-lemma}}(\alpha_{t - 1}, \delta_{t - 1}, \eta_{t - 1}, t-1)$, and $n_{t - 1} = (1 + \tilde\eps)^2 \alpha np$. Take $C' = C_{\ref{cor:regular-to-typical}}(\alpha, \tilde\eps, (1 - \tilde\eps)^2 \eta / (1 + \tilde\eps)^2)$, $\eps' = \min\{ \eps_{t - 1}, \tilde\eps, \eps_{0_{\ref{cor:regular-to-typical}}}(\alpha, \tilde\eps, t - 1) \}$, and $C = \max\{ C', C_{t - 1}\} /((1 - \tilde \eps)^{2} \alpha \eta)$. Finally, let $\eps_0 = \min\{\eps_{t - 1}, \eps_{0_{\ref{cor:regular-to-typical}}}(\alpha, \eps', \eta, t)\}$ and $c'_1 = c_{\ref{cor:regular-to-typical}}(\alpha, \eps', \eta, t)$, $c'_2 = c_{\ref{cor:regular-to-typical}}(\alpha, \tilde \eps, (1 - \tilde\eps)^2 \eta / (1 + \tilde\eps)^2, t-1)$, $c'_3 = c_{\ref{prop:counting-cliques}}(\eps', t)$. From now on we assume that the lemma holds when applied for $t - 1$, $\alpha_{t - 1}$, $\delta_{t - 1}$, $\eta_{t - 1}$, and $n_{t - 1}$. Since this happens with probability at least
  \[
    1 - e^{-c_{t-1} \cdot n_{t-1}^2 p^{2 (t- 3) + 1} } = 1 - e^{-c_{t-1} \cdot (1 + \tilde \eps)^2 \alpha^2 n^2 p^{2 (t- 2) + 1} },
  \]
  it is sufficient to show that the induction step holds with probability at least $1 - e^{- \Omega( n^2 p^{2 (t- 2) + 1})}$ with the hidden constant depending only on $\delta, \alpha$, and $\eta$, and then set $c$ to be sufficiently small with respect to that constant. We further assume that $\GNp$ is such that the conclusions of Proposition~\ref{prop:counting-cliques} for $t$, $\eps'$ as $\eps$, and $\eps' n_1$ as $n$, Corollary~\ref{cor:regular-to-typical} for $\alpha$, $\eps'$, and $\eta$, and Corollary~\ref{cor:regular-to-typical} for $\alpha$, $\tilde\eps$ as $\eps'$, $(1 - \tilde\eps)^2\eta / (1 + \tilde\eps)^2$ as $\eta$, and $(1 + \tilde\eps)^2 \alpha np$ as $n$, hold. This happens with probability at least
  \[
    1 - e^{-c'_1 n^2 p} - e^{-c'_2 \cdot (1 + \tilde\eps)^4 \alpha^2 \eta^2 n^2 p^3} - e^{-c'_3 \cdot \eps'^2 n^2 p^{t - 1}} = 1 - \eps^{-\Omega(n^2p^{2(t - 2) + 1})}.
  \]

  Observe that our choice of $\eps_0$ and Corollary~\ref{cor:regular-to-typical} imply that any $G$ as in the lemma is $\eps'$-typical. Let $v \in V_1$ be an $\eps'$-typical vertex and $N_i := N(v, V_i)$ its neighbourhoods for $2 \leq i \leq t$. By the definition of typical vertices we know that the sets $N_i$ satisfy $|N_i| = (1 \pm \eps')d_{ij} n_ip$. From the assumptions on $n_i$ and $d_{ij}$ (in the definition of the graph class $\cG(K_{t}, n, \eta, \alpha, \eps, p$)), the assumption on $n$ in the statement of the lemma, the choice of $C'$, and the fact that $\eps' \leq \tilde \eps$, we deduce that the sets $N_i$ satisfy
  \begin{align*}
    (1 + \tilde\eps)^2 \alpha np \geq |N_i| \ge (1 - \tilde\eps)^2 \alpha \eta np &\ge \max\{ C'p^{-t + 1}, C'p^{-t + 2}\log{N} \} \\
    &\ge \max\{ C'p^{-2}, C'p^{-1}\log{N} \}.
  \end{align*}
  Furthermore, we know that $(N_2, \ldots, N_t)$ is an $(\eps', p)$-regular $(t - 1)$-tuple with densities $(1 \pm \eps') d_{ij} p \subseteq (1 \pm \tilde\eps) d_{ij}p$ between $N_i$ and $N_j$.
  Hence, Corollary~\ref{cor:regular-to-typical} applied for $\alpha$, $\tilde\eps$ as $\eps'$, $(1 - \tilde\eps)^{2}\eta/(1 + \tilde\eps)^{2}$ as $\eta$, and $(1 + \tilde\eps)^2 \alpha np$ as $n$, shows that $(N_2, \ldots, N_{t})$ is $\tilde \eps$-typical. By induction hypothesis for $\delta'$, it follows that $v$ belongs to
  \[
    (1 \pm \delta') \bigg(\prod_{2 \leq i \leq t} |N_i| \bigg) \bigg(\prod_{2 \leq i < j \leq t} (1 \pm \tilde \eps) d_{ij}p \bigg) = (1 \pm \delta') (1 \pm \tilde \eps)^{\binom{t}{2}} \bigg(\prod_{2 \leq i \leq t} n_i \bigg) \bigg(\prod_{1 \leq i < j \leq t} d_{ij} p \bigg)
  \]
  canonical copies of $K_{t}$ in $(V_1, \ldots, V_{t})$. This settles the lower bound, as there are at least $(1 - \eps') n_1$ vertices $v \in V_1$ which are $\eps'$-typical, and by our choice of $\delta', \eps'$, and $\tilde\eps$.

  Moreover, as there are at most $\eps' n_1$ vertices $u \in V_1$ which are not $\eps'$-typical, by Proposition~\ref{prop:counting-cliques} applied for $t$, $\eps'$ as $\eps$, and $\eps' n_1$ as $n$, we get that they in total belong to at most $2\eps'n_1 \big( \prod_{2 \leq i \leq t} n_i \big) p^{e(K_{t})}$ such copies. Therefore, the upper bound for the number of canonical copies is given by
  \begin{align*}
    (1 + \delta') (1 + \tilde\eps)^{\binom{t}{2}} \bigg(\prod_{1 \leq i \leq t} n_i \bigg) \bigg(\prod_{1 \leq i < j \leq t} d_{ij} p \bigg) &+ 2\eps' \bigg(\prod_{1 \leq i \leq t} n_i \bigg) p^{e(K_{t})} \\
    &\leq (1 + \delta) \bigg(\prod_{1 \leq i \leq t} n_i \bigg) \bigg(\prod_{1 \leq i < j \leq t} d_{ij} p \bigg)
  \end{align*}
  again by our choice of $\delta'$, $\eps'$, and $\tilde\eps$.
\end{proof}

The definition of $\eps$-typical tuples states that the induced neighbourhoods of most of the vertices are of the right size and inherit regularity. We now generalise this idea from vertices to subgraphs. However, in order to simplify the notation and as this suffices for our purposes, we only consider the induced neighbourhoods in two sets.
\begin{definition}\label{def:typical-cliques}
  Let $t \geq 3$ be an integer and $(V_1, \ldots, V_t)$ a $t$-tuple with densities $d_{ij}p$ between $V_i$ and $V_j$. A canonical copy of $K_{t - 2}$ in $(V_1, \ldots, V_{t - 2})$ is said to be $\eps$-{\em typical} if:
  \begin{enumerate}[(i), font=\itshape]
    \item\label{typical-neighbourhoods-size} $|N_i| = (1 \pm \eps) n_i \big( \prod_{1 \leq j \leq t - 2} d_{ij} p \big)$, and
    \item\label{typical-neighbourhoods-regular} $(N_{t - 1}, N_t)$ is an $(\eps, p)$-regular pair of density $(1 \pm \eps) d_{(t - 1)t} p$,
  \end{enumerate}
  where $N_i := N(V(K_{t - 2}), V_i)$, for $i \in \{t - 1, t\}$.
\end{definition}
The next lemma provides a lower bound for the number of typical canonical copies of $K_{t - 2}$ in regular $t$-tuples. Combining it with the upper bound on the number of canonical copies given by Lemma~\ref{lem:counting-lemma} one easily sees that in a random graph $\GNp$ w.h.p.\ almost all possible canonical copies of $K_{t - 2}$ in any regular $t$-tuple are actually typical.
\begin{lemma}\label{lem:typical-cliques-counting}
  Let $t \geq 3$ be an integer. Then for all $\alpha, \delta, \eta > 0$, there exist positive constants $\eps_0(\alpha, \delta, \eta, t)$, $C(\alpha, \delta, \eta, t)$, and $c(\alpha, \delta, \eta, t)$ such that for every integer $n \geq \max\{Cp^{-t + 1}, Cp^{-t + 2}\log{N}\}$ a random graph $\GNp$ satisfies the following with probability at least $1 - e^{-c \cdot n^2 p^{2(t-3)+1}}$. Every subgraph $G$ of $\GNp$ in $\cG(K_t, n, \eta, \alpha, \eps, p)$ contains at least
  \[
    (1 - \delta) \bigg(\prod_{1 \le i \le t - 2} n_i \bigg) \bigg( \prod_{1 \le i < j \le t - 2} d_{ij} p \bigg)
  \]
  canonical copies of $K_{t - 2}$ in $(V_1, \ldots, V_{t - 2})$ which are $\delta$-typical, provided that $\eps \leq \eps_0$.
\end{lemma}
\begin{proof}
  The proof is by induction on $t$. For $t = 3$ the assertion of the lemma follows directly from the definition of typical tuples. So consider some $t \geq 4$.

  For given $\delta, \alpha, \eta$, we choose the constants similarly as in the previous proof: choose $\delta'$ and $\tilde\eps$ such that $(1 \pm \delta')^{2} \subseteq (1 \pm \delta)$ and $(1 \pm \tilde \eps)^{t^{2}} \subseteq (1 \pm \delta')$. Furthermore, we define $\alpha_{t - 1} = \alpha$, $\delta_{t - 1} = \delta'$, $\eta_{t - 1} = (1 - \tilde\eps)^{2}\eta/(1 + \tilde\eps)^{2}$. Let $\eps_{t - 1} = \eps_{0_{\ref{lem:counting-lemma}}}(\alpha_{t - 1}, \delta_{t - 1}, \eta_{t - 1}, t- 1)$, $C_{t - 1} = C_{\ref{lem:counting-lemma}}(\alpha_{t - 1}, \delta_{t - 1}, \eta_{t - 1}, t-1)$, $c_{t - 1} = c_{\ref{lem:counting-lemma}}(\alpha_{t - 1}, \delta_{t - 1}, \eta_{t - 1}, t-1)$, and $n_{t - 1} = (1 + \tilde\eps)^2 \alpha np$. Take $C' = C_{\ref{cor:regular-to-typical}}(\alpha, \tilde\eps, (1 - \tilde\eps)^2 \eta / (1 + \tilde\eps)^2, t - 1)$, $\eps' = \min\{\eps_{t - 1},\tilde\eps, \eps_{0_{\ref{cor:regular-to-typical}}}(\alpha, \tilde\eps, \eta, t - 1)\}$, and $C = \max\{ C', C_{t -1} \}/((1 - \tilde\eps)^{2} \alpha \eta)$. Finally, let $\eps_0 = \eps_{0_{\ref{cor:regular-to-typical}}}(\alpha, \eps', \eta, t)$, $c'_1 = c_{\ref{cor:regular-to-typical}}(\alpha, \eps', \eta, t)$, and $c'_2 = c_{\ref{cor:regular-to-typical}}(\alpha, \tilde\eps, (1 - \tilde\eps)^2 \eta / (1 + \tilde\eps)^2, t-1)$. From now on we assume that the lemma holds when applied for $t - 1$, $\alpha_{t - 1}$, $\delta_{t - 1}$, $\eta_{t - 1}$, and $n_{t - 1}$. Since this happens with probability at least
  \[
    1 - e^{-c_{t-1} \cdot n_{t-1}^2 p^{2 (t- 4) + 1} } = 1 - e^{-c_{t-1} \cdot (1 + \tilde\eps)^2 \alpha^2 n^2 p^{2 (t- 3) + 1} }
  \]
  it is sufficient to show that the induction step holds with probability at least $1 - e^{-\Omega( n^2 p^{2 (t- 3) + 1})}$ with the hidden constant depending only on $\delta, \alpha$, $\eta$ and $t$, and then set $c$ to be sufficiently small with respect to that constant. We further assume that $\GNp$ is such that the conclusion of Corollary~\ref{cor:regular-to-typical} holds both for $\alpha$, $\eps'$, and $\eta$, as well as for $\alpha$, $\tilde\eps$ as $\eps'$, $(1 - \tilde\eps)^2\eta / (1 + \tilde\eps)^2$ as $\eta$, and $(1 + \tilde\eps)^2 \alpha np$ as $n$. This happens with probability at least
  \[
    1 - e^{-c'_1 n^2 p} - e^{-c'_2 \cdot (1 + \tilde\eps)^4 \alpha^2 \eta^2 n^2 p^3} = 1 - e^{-\Omega(n^2 p^{2 (t - 3) + 1})}.
  \]

  Let $G$ be as in the lemma and note that by Corollary \ref{cor:regular-to-typical} we have that $G$ is $\eps'$-typical. Let $v \in V_1$ be an $\eps'$-typical vertex and $N_i := N(v, V_i)$ its neighbourhoods for $2 \leq i \leq t$. Similarly as in the previous lemma we have
  \begin{align*}
    (1 + \tilde\eps)^2 \alpha np \geq |N_i| \ge (1 - \tilde\eps)^2 \alpha \eta np &\ge \max\{C'p^{-t + 2}, C'p^{-t + 3}\log{N}\} \\
    &\ge \max\{C'p^{-2}, C'p^{-1}\log{N}\},
  \end{align*}
  as $\eps' \leq \tilde\eps$. Furthermore, $(N_2, \ldots, N_t)$ is an $(\eps', p)$-regular $(t - 1)$-tuple with densities $(1 \pm \eps') d_{ij} p \subseteq (1 \pm \tilde\eps)d_{ij}p$ between $N_i$ and $N_j$, hence it must be $\tilde\eps$-typical due to Corollary~\ref{cor:regular-to-typical} applied for $t - 1$, $\alpha$, $\tilde \eps$ as $\eps'$, $(1 - \tilde\eps)^{2}\eta/(1 + \tilde\eps)^{2}$ as $\eta$, and $(1 + \tilde\eps)^2 \alpha np$ as $n$. By induction hypothesis for $\delta'$, we obtain that there are at least
  \[
    (1 - \delta') \bigg(\prod_{2 \leq i \leq t - 2} |N_i| \bigg) \bigg(\prod_{2 \leq i < j \leq t - 2} (1 - \tilde\eps)d_{ij} p \bigg) \geq (1 - \delta') \bigg( \prod_{2 \leq i \leq t - 2} n_i \bigg) \bigg( \prod_{1 \leq i < j \leq t - 2} (1 - \tilde\eps)d_{ij} p \bigg)
  \]
  $\delta'$-typical canonical copies of $K_{t - 3}$ in $(N_2, \ldots, N_{t - 2})$. As there are at least $(1 - \eps') n_1$ vertices in $V_1$ which are $\eps'$-typical and as $\eps' \le \tilde\eps$, it follows that there are at least
  \[
    (1 - \delta')(1 - \tilde\eps)^{t^{2}} \bigg(\prod_{1 \leq i \leq t - 2} n_i \bigg) \bigg(\prod_{1 \leq i < j \leq t - 2} d_{ij} p \bigg) \geq (1 - \delta) \bigg(\prod_{1 \leq i \leq t - 2} n_i \bigg) \bigg(\prod_{1 \leq i < j \leq t - 2} d_{ij} p \bigg)
  \]
  canonical copies of $K_{t - 2}$ in $(V_1, \ldots, V_{t - 2})$ which have a common neighbourhood in $V_{i}$, for $i \in \{t - 1, t\}$, of size
  \[
    (1 \pm \delta') |N_{i}| \bigg(\prod_{2 \leq j \leq t - 2} (1 \pm \tilde\eps) d_{ij} p \bigg) = (1 \pm \delta')^{2} n_i \bigg(\prod_{1 \leq j \leq t - 2} d_{ij} p \bigg) = (1 \pm \delta) n_i \bigg(\prod_{1 \leq j \leq t - 2} d_{ij} p \bigg).
  \]
 Since common neighbourhoods in $N_{t - 1}$ and $N_{t}$ of obtained copies of $K_{t - 3}$ span a $(\delta', p)$-regular pair of density $(1 \pm \delta') d_{(t - 1)t} p$ by induction hypothesis, it follows that all obtained copies of $K_{t - 2}$ have common neighbourhoods in $V_{t - 1}$ and $V_{t}$ which span a $(\delta, p)$-regular pair of density $(1 \pm \delta) d_{(t - 1)t} p$. Hence, all such copies are $\delta$-typical. This completes the proof of the lemma.
\end{proof}

We combine the properties of $t$-tuples given by the previous two lemmas into the following definition for the ease of further reference.
\begin{definition}[$(\delta, \eps)$-{\bf super-typical}]\label{def:super-typical}
  For an integer $t \geq 3$ we say that an $\eps$-typical $t$-tuple $(V_1, \ldots, V_t)$ is $(\delta, \eps)$-\emph{super-typical} if:
  \begin{enumerate}[(i), font=\itshape]
    \item\label{def:st-counting-lemma-small} there are
    \[
      (1 \pm \delta) \bigg( \prod_{2 \leq i \leq t - 1} |V_i| \bigg) \bigg( \prod_{2 \leq i < j \leq t - 1} d_{ij}p \bigg)
    \]
    canonical copies of $K_{t - 2}$ in $(V_2, \ldots, V_{t - 1})$,

    \item\label{def:st-counting-lemma-large} there are
    \[
      (1 \pm \delta) \bigg( \prod_{1 \leq i \leq t - 1} |V_i| \bigg) \bigg( \prod_{1 \leq i < j \leq t - 1} d_{ij}p \bigg) \quad \text{and} \quad (1 \pm \delta) \bigg( \prod_{2 \leq i \leq t} |V_i| \bigg) \bigg( \prod_{2 \leq i < j \leq t} d_{ij}p \bigg)
    \]
    canonical copies of $K_{t - 1}$ in $(V_1, \ldots, V_{t - 1})$ and $(V_2, \ldots, V_t)$, respectively, and

    \item\label{def:st-typical-cliques} there are at least
    \[
      (1 - \delta) \bigg( \prod_{2 \leq i \leq t - 1} |V_i| \bigg) \bigg( \prod_{2 \leq i < j \leq t - 1} d_{ij}p \bigg)
    \]
    canonical copies of $K_{t - 2}$ in $(V_2, \ldots, V_{t - 1})$ which are $\delta$-typical with respect to $V_1$ and $V_t$.
  \end{enumerate}
\end{definition}
In other words, the definition of super-typical tuples requires that almost all canonical copies of $K_{t - 2}$ in $(V_2, \ldots, V_{t - 1})$ have common neighbourhoods in $V_1$ and $V_t$ which are of the right size and form a $(\delta, p)$-regular pair. We illustrate a $(\delta, \eps)$-super-typical tuple in Figure~\ref{fig:super-typical} below.
\begin{figure}[!htbp]
  \centering
  \begin{subfigure}[t]{0.45\textwidth}
    \centering
    \includegraphics[scale=0.5]{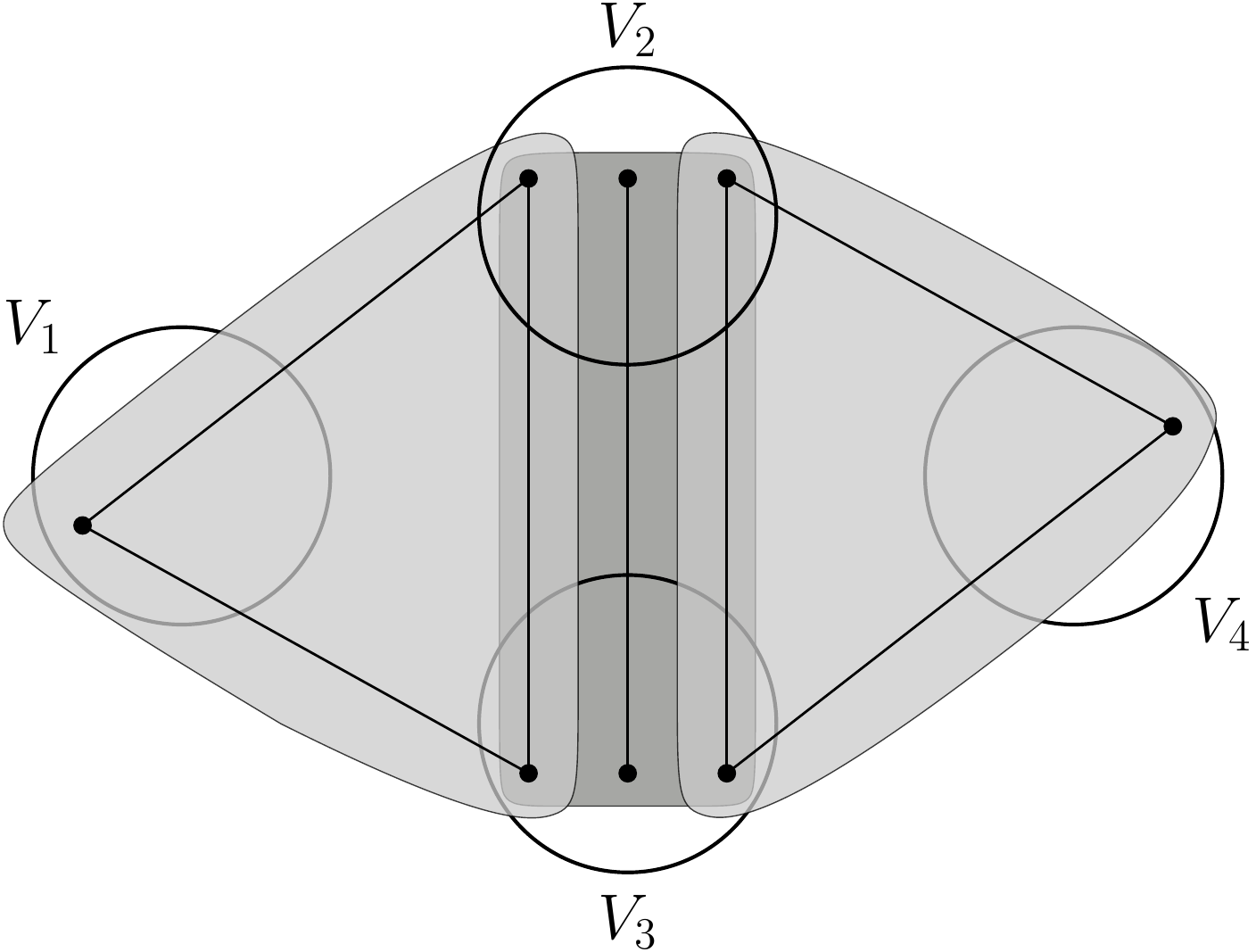}
    \caption{The number of copies of $K_3$ in $(V_1, V_2, V_3)$ and $(V_2, V_3, V_4)$ as well as $K_2$ in $(V_2, V_3)$ is roughly what one would expect in a random graph with pairs $(V_i, V_j)$ of density $d_{ij}p$.}
    \label{fig:super-typical-a}
  \end{subfigure}
  \hspace{2em}
  \begin{subfigure}[t]{0.45\textwidth}
    \centering
    \includegraphics[scale=0.5]{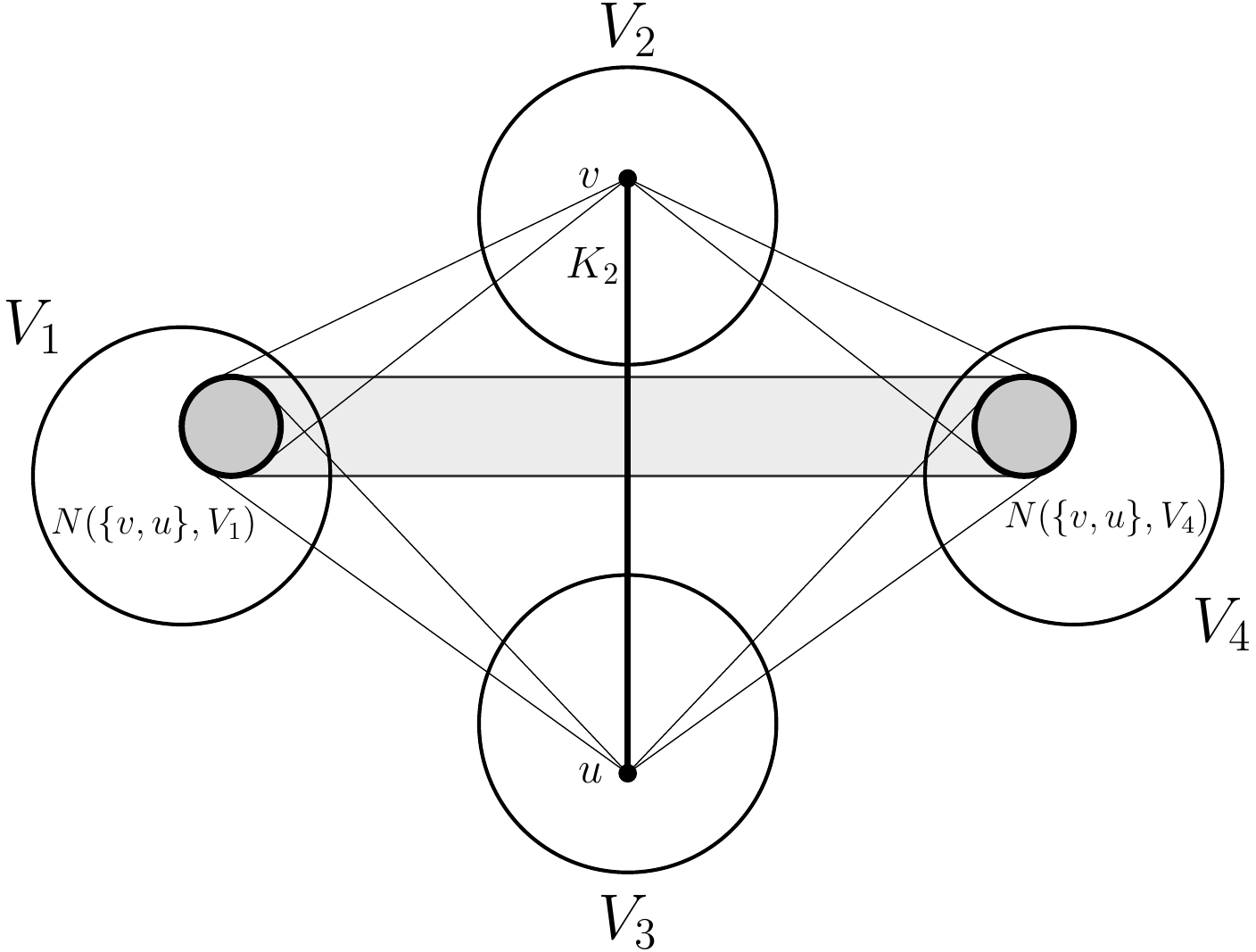}
    \caption{Common neighbourhoods in $V_1$ and $V_4$ of most of the copies of $K_2$ in $(V_2, V_3)$ are of size roughly as expected in a random graph with pairs $(V_i, V_j)$ of density $d_{ij}p$ and are in addition $(\delta, p)$-regular.}
    \label{fig:super-typical-b}
  \end{subfigure}
  \caption{$(\delta, \eps)$-super-typical $4$-tuple}
  \label{fig:super-typical}
\end{figure}

The notion from above allows us to claim that every time we encounter a $(\delta, \eps)$-super-typical $t$-tuple, the vertices of most of the copies of $K_{t - 1}$ in $(V_1, \ldots, V_{t - 1})$ have at least one common neighbour in $V_{t}$, as the induced copy of $K_{t - 2}$ in $(V_2, \ldots, V_{t - 1})$ is typically such that its neighbourhoods in $V_1$ and $V_t$ form a $(\delta, p)$-regular pair. This property turns out to be crucial in the embedding of a long $k$-cycle into $G$.

We conclude this section with a corollary of the previous two lemmas as well as Lemma~\ref{cor:regular-to-typical}, stating that every subgraph of $\GNp$ that belongs to the class $\cG(K_t, n, \eta, \alpha, \eps, p)$ is w.h.p.\ $(\delta, \eps')$-super-typical for sufficiently small $\eps$ compared to $\delta$ and $\eps'$.
\begin{corollary}\label{cor:regular-to-super-typical}
  Let $t \geq 3$ be an integer. Then for all $\alpha, \eps', \delta, \eta > 0$, there exist positive constants $\eps_0(\alpha, \eps', \delta, \eta, t)$, $C(\alpha, \eps', \delta, \eta, t)$, and $c(\alpha, \eps', \delta, \eta, t)$ such that for every integer $n \geq \max\{ Cp^{-t + 1}, Cp^{-t + 2}\log{N} \}$ a random graph $\GNp$ satisfies the following with probability at least $1 - e^{-c \cdot n^2 p^{2(t-3) + 1}}$. Every subgraph $G$ of $\GNp$ in $\cG(K_{t}, n, \eta, \alpha, \eps, p)$ is $(\delta, \eps')$-super-typical, provided that $\eps \leq \eps_0$.
\end{corollary}
\begin{proof}
  If $t = 3$ then just by choosing $\eps_0 = \eps_{0_{\ref{cor:regular-to-typical}}}(\alpha, \min\{\delta, \eps'\}, \eta, 3)$, $C = C_{\ref{cor:regular-to-typical}}(\alpha, \min\{\delta, \eps'\}, \eta, 3)$, $c = c_{\ref{cor:regular-to-typical}}(\alpha, \min\{\delta, \eps'\}, \eta),3$ and applying Corollary~\ref{cor:regular-to-typical} to $G \in \cG(K_3, n, \eta, \alpha, \eps, p)$ with $t = 3$, $\alpha$, $\delta$ as $\eps'$, and $\eta$, we get that $G$ is $(\delta, \eps')$-super-typical with probability at least $1 - e^{-c n^2 p}$.
   Otherwise, for $t \geq 4$, given $\alpha, \delta, \eta$, let
  \[
    \eps_0 = \min\{ \eps_{0_{\ref{cor:regular-to-typical}}}(\alpha, \eps', \eta, t), \eps_{0_{\ref{lem:counting-lemma}}}(\alpha, \delta,\eta, t - 2), \eps_{0_{\ref{lem:counting-lemma}}}(\alpha, \delta, \eta, t - 1), \eps_{0_{\ref{lem:typical-cliques-counting}}}(\alpha, \delta, t) \},
  \]
  \[
    C = \max\{ C_{\ref{cor:regular-to-typical}}(\alpha, \eps', \eta, t), C_{\ref{lem:counting-lemma}}(\alpha, \delta, \eta, t- 2),
    C_{\ref{lem:counting-lemma}}(\alpha, \delta, \eta, t- 1), C_{\ref{lem:typical-cliques-counting}}(\alpha, \delta, \eta, t) \},
  \]
  and take $c'_1 = c_{\ref{cor:regular-to-typical}}(\alpha, \eps', \eta, t)$, $c'_2 = c_{\ref{lem:counting-lemma}}(\alpha, \delta, \eta, t-2)$, $c'_3 = c_{\ref{lem:counting-lemma}}(\alpha, \delta, \eta, t-1)$, $c'_4 = c_{\ref{lem:typical-cliques-counting}}(\alpha, \delta, \eta, t)$, and $c = \min\{ c'_1, c'_2, c'_3, c'_4 \}/10$. We assume that $\GNp$ is such that the conclusions of Corollary~\ref{cor:regular-to-typical}, Lemma~\ref{lem:counting-lemma} both for $t - 2$ and $t - 1$, and Lemma~\ref{lem:typical-cliques-counting}, hold. This happens with probability at least
  \[
    1 - e^{-c_1' \cdot n^2p} - e^{-c_2' \cdot n^2p^{2(t - 3) + 1}} - e^{-c_3' \cdot n^2p^{2(t - 4) + 1}} - e^{-c_4' \cdot n^2p^{2(t - 3) + 1}} \geq 1 - e^{-c \cdot n^2 p^{2(t - 3) + 1}}.
  \]

  Let $G$ be a subgraph of $\GNp$ which belongs to the class $\cG(K_t, n, \eta, \alpha, \eps, p)$. By Corollary~\ref{cor:regular-to-typical} and our choice of $\eps_0$, $G$ is $\eps'$-typical. Moreover, by Lemma~\ref{lem:counting-lemma} applied for $t - 1$ as $t$, the number of copies of $K_{t - 1}$ in both $(V_1, \ldots, V_{t - 1})$ and $(V_2, \ldots, V_t)$ is as required by Definition~\ref{def:super-typical}~$\ref{def:st-counting-lemma-large}$. Similarly, Lemma~\ref{lem:counting-lemma} applied this time for $t - 2$ as $t$ shows that the number of copies of $K_{t - 2}$ in $(V_2, \ldots, V_{t - 1})$ is again as required. Lastly, graph $G$ satisfies the assertion of Lemma~\ref{lem:typical-cliques-counting} applied for $t$, that is it contains the required number of $\delta$-typical canonical copies of $K_{t - 2}$ as in Definition~\ref{def:super-typical}~$\ref{def:st-typical-cliques}$.
\end{proof}

%% file: main.tex
\section{Proof of the Main Theorem}\label{sec:main-proof}

The overall structure of our proof follows the standard approach for embedding structures into sparse graphs. First we apply the sparse regularity lemma, more precisely, the version given by Corollary~\ref{cor:nice-regularity}, to obtain an $(\eps, p)$-regular partition of $V(G)$. Then we find a $k$-cycle on the partition classes using the result for dense graphs, i.e.\ Theorem~\ref{thm:KSS}. As usual, the main work consists of subsequently using this $k$-cycle on the partition classes to find an almost spanning $k$-cycle in the original graph. To achieve this, our main tool is a \textit{clique expansion lemma} (Lemma~\ref{lem:clique-expansion-lemma}) that roughly states the following. Suppose we have an $\ell$-tuple $(V_1, \ldots, V_{\ell})$ of sets of size $n$ so that for all $i \in [\ell - k]$ the tuple $(V_{i}, \ldots, V_{i + k})$ is $(\delta, \eps)$-super-typical. Then every $\delta$-fraction of the canonical copies of $K_k$ in $(V_{1}, \ldots, V_{k})$ contains a copy $K$ such that almost all canonical copies of $K_k$ in $(V_{\ell - k + 1}, \ldots, V_{\ell})$ can be reached by a $k$-path in $(V_1, \ldots, V_{\ell})$ that starts in $K$.

With such a clique expansion lemma at hand, the theorem can then be proven using a standard approach. We briefly explain the main idea. Assume that the value of $\ell$ (from the clique expansion lemma) is smaller than the number $t$ of partition classes in the $(\eps, p)$-regular partition of $V(G)$. Assume furthermore that the partition classes $V_1, \ldots, V_t$ form a $k$-cycle, that is $(V_i, V_j)$ is an $(\eps, p)$-regular pair of positive density whenever $|i - j| \leq k$. By choosing constants appropriately and conditioning on the fact that the underlying random graph $\GNp$ is `nice'---most notably satisfying Corollary~\ref{cor:regular-to-super-typical}---we may assume that any tuple $(V'_{i}, \ldots, V'_{i + \ell - 1})$ with $V'_j \subseteq V_j$ and $|V'_j| = \eps |V_j|$ satisfies the assumption of the clique expansion lemma. By considering $(V'_{1}, \ldots, V'_{\ell})$ we thus find a $K \in K_k(V'_{1}, \ldots, V'_{k})$ that expands to almost all canonical copies of $K_k$ in $(V'_{\ell - k + 1}, \ldots, V'_{\ell})$. Now consider a tuple $(V'_{\ell - k + 1}, \ldots, V'_{2\ell - k})$. Repeating the expansion argument we see that there is a $k$-{\em path} from $K$ to some $K' \in K_k(V'_{\ell - k + 1}, \ldots, V'_{\ell})$ such that $K'$ expands to almost all canonical copies of $K_k$ in $(V'_{2\ell - 2k + 1}, \ldots, V'_{2\ell - k})$. Clearly, we can repeat this process until we have covered all but at most $\eps |V_i|$ vertices in each partition class $V_i$, thereby obtaining an almost spanning $k$-path. Moreover, the usual modifications of this approach allow us to actually find an almost spanning $k$-cycle.

Our plan for the rest of the section is as follows. In Section~\ref{sec:clique-expansion} we state and give a proof of the clique expansion lemma. Unfortunately, our proof requires that $\ell$ is logarithmic in $N$. This is where Lemma~\ref{lem:regularity-partitioning} comes into the game: starting from a basic regular partition with a constant number of classes (as given by the sparse regularity lemma) we can obtain a regular partition with logarithmically many partition classes, to which we can then apply the embedding strategy outlined above. The price we have to pay for extending the standard partition to one with logarithmically many classes is that we have to make $p$ slightly larger, in order to inherit regularity between the smaller classes.

\subsection{Clique expansion lemma}\label{sec:clique-expansion}
For an integer $\ell$ and a graph $G \in \cG(P_\ell^k, n, \eta, \alpha, \eps, p)$ on the vertex partition $V_1, \ldots, V_{\ell}$, we say that a subset $\cK_k' \subseteq K_k(V_i, \ldots, V_{i + k - 1})$ \emph{expands} to a subset $\cK_k'' \subseteq K_k(V_j, \ldots, V_{j + k - 1})$ for some $1 \leq i < j \leq \ell - k + 1$, if for every $H'' \in \cK_k''$ there exist $H' \in \cK_k'$ and a canonical copy of $P^k_{j + k - i}$ in $G[V_i, V_{i + 1}, \ldots, V_{j + k - 1}]$ which connects $H'$ with $H''$.

Denote by $\cG_\delta(P^{k}_{\ell}, n, \alpha, \eps, p) \subseteq \cG(P^{k}_{\ell}, n, 1, \alpha, \eps, p)$ the class of graphs for which $|V_i| = n$ and $(V_i, \ldots, V_{i + k})$ is $(\delta, \eps)$-super-typical for all $i \in [\ell - k]$. Recall that we have $\cG(P^{k}_{\ell}, n, 1, \alpha, \eps', p)\subseteq \cG(P^{k}_{\ell}, n, 1, \alpha, \eps, p)$ for all $\eps'\le \eps$. Corollary~\ref{cor:regular-to-super-typical} thus implies that, for $\eps'$ small enough, w.h.p.\ a random graph $\GNp$ is such that all graphs from $\cG(P^{k}_{\ell}, n, 1, \alpha, \eps', p)$ which are subgraphs $\GNp$ also belong to $\cG_\delta(P^{k}_\ell, n, \alpha, \eps, p)$. With this definition at hand we state the clique expansion lemma.
\begin{lemma}\label{lem:clique-expansion-lemma}
  Let $k \geq 2$ be an integer. For all $\alpha > 0$ and $0 < \delta < 1/(20k)$, there exist positive constants $\eps_0(\alpha, \delta, k)$, $C(\alpha, \delta, k)$, and $c(\alpha, \delta, k)$ such that for every $\ell$, where $3k^2 \log N \leq \ell \leq N$, and every integer $n \geq \max\{Cp^{-k}, Cp^{-k + 1}\log{N}\}$ a random graph $\GNp$ satisfies the following with probability at least $1-e^{-c n^2 p^{2(k-2) +1}}$. Every subgraph $G$ of $\GNp$ in $\cG_\delta(P_{\ell}^{k}, n, \alpha, \eps, p)$ with $\eps \leq \eps_0$, is such that for every $\cK \subseteq K_{k}(V_1, \ldots, V_k)$ of size $|\cK| \geq \delta n^{v(K_{k})} (\alpha p)^{e(K_{k})}$, there exists $K \in \cK$ which expands to at least $(1 - 20k\delta) n^{v(K_{k})} (\alpha p)^{e(K_{k})}$ canonical copies of $K_k$ in $(V_{\ell - k + 1}, \ldots, V_{\ell})$.
\end{lemma}

We split the proof into several lemmas which, together with some additional observations, allow us to show the intended statement. First we show that the definition of a $(\delta, \eps)$-super-typical $(k + 1)$-tuple $(V_1, \ldots, V_{k + 1})$ allows to show that any large enough fraction of the canonical copies of $K_k$ in $(V_1, \ldots, V_k)$ expands to roughly the same fraction of the canonical copies in $(V_2, \ldots, V_{k+1})$---even if $V_1$ is much smaller than the other sets $V_i$. We give a brief explanation of why we find such a statement useful for what is to come. Assume we are given a $(\delta, \eps)$-super-typical $(k + 1)$-tuple $(V_1, \ldots, V_{k + 1})$ with $|V_1| = \ldots = |V_{k + 1}|$. Since $(\eps, p)$-regularity is inherited in large subsets (see, Proposition~\ref{prop:large-subsets-inheritance}), by taking any subset $\tilde V_1 \subseteq V_1$ that is not too small and applying Corollary~\ref{cor:regular-to-super-typical} to the tuple $(\tilde V_1, V_2, \ldots, V_{k + 1})$ we get that it is also $(\delta', \eps')$-super-typical for appropriate constants $\delta'$ and $\eps'$. Thus, following the reasoning from above, in such a tuple starting from almost all canonical copies of $K_k$ in $(\tilde V_1, V_2, \ldots, V_k)$ we expand in one step to almost all canonical copies of $K_k$ in $(V_2, \ldots, V_{k + 1})$. It is important to note that in the following lemma we do not require that the sets $V_i$ are of the same size.
\begin{lemma}\label{lem:one-step-expansion}
  Let $k \geq 2$ be an integer. Then for all $\alpha, \delta, \kappa > 0$, the following holds. Let $(V_1, \ldots, V_{k + 1})$ be a $(\delta, \eps)$-super-typical tuple with densities $d_{ij}p$ between $V_i$ and $V_j$. Then any set of at least $\kappa (\prod_{i = 1}^{k} |V_i|) (\prod_{1 \leq i < j \leq k} d_{ij}p)$ canonical copies of $K_k$ in $(V_1, \ldots, V_k)$ expands to at least $(\kappa - 3\kappa\delta - 6\delta) (\prod_{i = 1}^{k} |V_i|) (\prod_{2 \leq i < j \leq k + 1} d_{ij} p)$ canonical copies of $K_k$ in $(V_{2}, \ldots, V_{k + 1})$.
\end{lemma}
\begin{proof}
  The proof is solely based on the properties of super-typical tuples. More precisely, we use that the number of all canonical copies of $K_k$ and $K_{k - 1}$ is within a factor of $1 \pm \delta$ of their expectation and that all but a tiny fraction of the canonical copies of $K_{k - 1}$ in $(V_2,\ldots,V_k)$ are $\delta$-typical.

  Firstly, observe that Definition~\ref{def:typical-cliques}~$\ref{typical-neighbourhoods-size}$ implies a lower bound on the number of canonical copies of $K_k$ in $(V_1, \ldots, V_k)$ that a single $\delta$-typical copy of $K_{k-1}$ in $(V_2, \ldots, V_k)$ contributes to. Together with the lower bound on the number of $\delta$-typical copies in $(V_2, \ldots, V_k)$ given by Definition~\ref{def:super-typical}~$\ref{def:st-typical-cliques}$, we conclude that there are at least
  \[
    (1 - \delta) \bigg( \prod_{i = 2}^{k} |V_i| \bigg) \bigg( \prod_{2 \leq i < j \leq k} d_{ij}p \bigg) \cdot (1 - \delta) |V_1| \bigg( \prod_{i = 2}^{k} d_{1i}p \bigg) \geq (1 - 2\delta) \bigg( \prod_{i = 1}^{k} |V_i| \bigg) \bigg( \prod_{1 \leq i < j \leq k} d_{ij}p \bigg)
  \]
  copies of $K_k$ in $(V_1, \ldots, V_k)$ which have a $\delta$-typical copy of $K_{k - 1}$ as an induced subgraph. Combining it with the upper bound on the number of copies of $K_k$ in $(V_1, \ldots, V_k)$ given by Definition~\ref{def:super-typical}~$\ref{def:st-counting-lemma-large}$, we deduce that there are at most $3\delta (\prod_{i = 1}^{k} |V_i|) (\prod_{1 \leq i < j \leq k} d_{ij} p)$ copies of $K_k$ in $(V_1, \ldots, V_k)$ which do not have a $\delta$-typical copy of $K_{k - 1}$ as an induced subgraph.

  Let now $\cK \subseteq K_k(V_1, \ldots, V_k)$ be an arbitrary set of size $\kappa (\prod_{i = 1}^{k} |V_i|) (\prod_{1 \leq i < j \leq k} d_{ij} p)$ and let $\tilde \cK$ be the set of all
  $\delta$-typical canonical copies of $K_{k - 1}$ in $(V_2, \ldots, V_k)$ which belong to at least $2\delta |V_1| (\prod_{i = 2}^{k} d_{1i}p)$ copies of $K_k$ from $\cK$. In order to derive a lower bound on $|\tilde \cK|$ note that by the argument from above it follows that at least
  \[
    (\kappa - 3\delta) \bigg( \prod_{i = 1}^{k} |V_i| \bigg) \bigg( \prod_{1 \leq i < j \leq k} d_{ij} p \bigg)
  \]
  copies of $K_k$ in $\cK$ have an induced $\delta$-typical copy of $K_{k - 1}$ in $(V_2, \ldots, V_k)$. Furthermore, the number of copies of $K_k$ in $\cK$ that contain a canonical copy of $K_{k - 1}$ which is {\em not} contained in at least $2\delta |V_1| (\prod_{i = 2}^{k} d_{1i}p)$ copies of $K_k$ from $\cK$ is bounded by
  \[
    2\delta |V_1| \bigg( \prod_{i = 2}^{k} d_{1i}p \bigg) \cdot (1 + \delta) \bigg( \prod_{i = 2}^{k} |V_i| \bigg) \bigg( \prod_{2 \leq i < j \leq k} d_{ij}p \bigg)
  \]
  where the second factor comes from the upper bound on the total number of canonical copies of $K_{k - 1}$ in $(V_2, \ldots, V_{k})$. Finally, by Definition~\ref{def:typical-cliques}~$\ref{typical-neighbourhoods-size}$, each $\delta$-typical copy of $K_{k - 1}$ is contained in at most $(1 + \delta) |V_1| (\prod_{i = 2}^{k} d_{1i}p)$ canonical copies of $K_k$, and hence it holds that
  \begin{eqnarray*}
    |\tilde \cK| &\geq& \frac{(\kappa - 3\delta)(\prod_{i = 1}^{k} |V_i|) (\prod_{1 \leq i < j \leq k} d_{ij} p) - 2\delta(1 + \delta) (\prod_{i = 1}^{k} |V_i|) (\prod_{1 \leq i < j \leq k} d_{ij} p)}{(1 + \delta) |V_1| (\prod_{i = 2}^{k} d_{1i}p)} \\
    &\geq& (\kappa - \kappa\delta - 6\delta) \bigg( \prod_{i = 2}^{k} |V_i| \bigg) \bigg( \prod_{2 \leq i < j \leq k} d_{ij} p \bigg).
  \end{eqnarray*}
  Consider an arbitrary $K \in \tilde \cK$ and let $N_1 := N(K, V_1)$ and $N_{k+1} := N(K, V_{k+1})$. By Definition~\ref{def:typical-cliques}~$\ref{typical-neighbourhoods-regular}$ we know that $(N_1, N_{k + 1})$ is a $(\delta, p)$-regular pair of density $(1 \pm \delta)d_{1(k + 1)}p$. Let $N_1'$ denote the vertices in $V_1$ which together with $K$ form a copy of $K_{k}$ which belongs to $\cK$. By our choice of the set $\tilde\cK$ we know that $|N_1'| \ge 2\delta |V_1| (\prod_{i = 2}^{k} d_{1i}p)$ and hence also $|N_1'| \geq \delta|N_1|$. Thus, there is an edge between $N_{1}'$ and any subset of $N_{k + 1}$ of size $\delta |N_{k + 1}|$. Consequently, $\cK$ expands to at
  least
  \[
    |\tilde \cK| (1 - \delta)|N_{k + 1}| \geq (\kappa - 3\kappa\delta - 6\delta) \bigg(\prod_{i = 2}^{k + 1} |V_i|\bigg) \bigg(\prod_{2 \leq i < j \leq k + 1} d_{ij} p \bigg),
  \]
  canonical copies of $K_k$ in $(V_2, \ldots, V_{k + 1})$. This completes the proof.
\end{proof}

The next lemma is the most important step in the proof of the clique expansion lemma. It roughly states that w.h.p.\ every subgraph $G$ of $\GNp$ that belongs to the class $\cG_\delta(P^k_{2k}, n, \alpha, \eps, p)$ is such that every small linear fraction of the canonical copies of $K_k$ in $(V_1, \ldots, V_k)$ expands to almost all copies of $K_k$ in $(V_{k + 1}, \ldots, V_{2k})$. For the convenience of the reader we first illustrate the proof for $k = 2$. Suppose we are given a graph $G \in \cG_\delta(P_{4}^{2}, n, \alpha, \eps, p)$. For any large enough set of edges $E'\subseteq E(V_1, V_2)$, properties of $(\eps, p)$-regular pairs imply that there is a not too small subset $\tilde V_2 \subseteq V_2$ so that every vertex $v \in \tilde V_2$ is incident to many edges in $E'$. As the neighbourhoods $N(v, V_1)$ and $N(v, V_3)$ span a regular pair, this shows that edges in $N_{E'}(v, V_1)$ are contained in many triangles with $v$ and $N(v, V_3)$. In particular, $E'$ expands to almost all edges in $(\tilde V_2, V_3)$. We may thus apply Lemma~\ref{lem:one-step-expansion} to $(\tilde V_2, V_3, V_4)$ to obtain the claimed expansion. The proof for larger values of $k$ proceeds inductively using similar techniques, however it is technically significantly more involved.
\begin{lemma}\label{lem:main-expansion}
  Let $k \geq 2$ be an integer. Then for all $0<\alpha, \delta <1$, there exist positive constants $\eps_0(\alpha, \delta, k)$, $C(\alpha, \delta, k)$, and $c(\alpha, \delta, k)$ such that for every integer $n \geq \max\{ Cp^{-k}, Cp^{-k + 1}\log N \}$ a random graph $\GNp$ satisfies the following with probability at least $1 - e^{-c n^2 p^{2(k-2) +1}}$. Every subgraph $G$ of $\GNp$ in $\cG_\delta(P^{k}_{2k}, n, \alpha, \eps, p)$, with $\eps \leq \eps_0$, is such that every set of $\delta n^{v(K_k)}(\alpha p)^{e(K_k)}$ canonical copies of $K_k$ in $(V_1, \ldots, V_k)$ expands to at least $(1 - 10\delta)n^{v(K_k)}(\alpha p)^{e(K_k)}$ canonical copies of $K_k$ in $(V_{k + 1}, \ldots, V_{2k})$.
\end{lemma}
\begin{proof}
  To shorten notation, we write $p_0 := \alpha p$ and $x_k := n^{v(K_k)}p_0^{e(K_k)}$ in the remainder of the proof. Given $\alpha$ and $\delta$, let $\delta' = \delta/32$, $\eps' = \min\{(\frac{\alpha}{\alpha + 1})^2 , (\frac{\delta'}{4k})^{4}, (\frac{16k}{4^{k}})^{2}\}$, $C_1 = \max\{ C_{\ref{cor:regular-to-super-typical}}(\alpha, \eps', \delta', \eps', k + 1), C_{\ref{cor:regular-to-super-typical}}(\alpha, \eps', \delta', (1 - \eps')^2/(1 + \eps')^2, k) \}$, $c_1 = \min \{ c_{\ref{cor:regular-to-super-typical}}(\alpha, \eps', \delta', \eps', k + 1), c_{\ref{cor:regular-to-super-typical}}(\alpha, \eps', \delta', (1 - \eps')^2/(1 + \eps')^2, k) \}$. If $k \geq 3$ we furthermore define $\eps_{k - 1} = \eps_{0_{\ref{lem:main-expansion}}}(\alpha, \delta', k - 1)$, $C_{k - 1} = C_{\ref{lem:main-expansion}}(\alpha, \delta', k - 1)$, and $c_{k - 1} = c_{\ref{lem:main-expansion}}(\alpha, \delta', k - 1)$. Lastly, take $\eps_0 = \min\{ \eps_{k - 1}, \eps'^2, \eps_{0_{\ref{cor:regular-to-super-typical}}}^{2} (\alpha, \eps', \delta', \eps', k + 1) \}$ and $C = \max\{ C_{k - 1}, C_1 \} / ((1 - \eps')^2\alpha)$, where $\eps_1 = \eps'$.

  We first consider the case $k = 2$ and assume that $\GNp$ is such that the conclusion of Corollary~\ref{cor:regular-to-super-typical} holds for $t = 3$, $\alpha$, $\eps'$, $\delta$, and $\eps'$ as $\eta$. This happens with probability at least $1 - e^{-c_1 \cdot n^2 p}$. Denote by $E'$ an arbitrary subset of $E(V_1, V_2)$ of size $\delta x_2$. Let $\tilde V_2$ be the subset of $\eps$-typical vertices $v \in V_2$ for which $|N_{E'}(v, V_1)| \geq (\delta/2) np_0$. Recall, there are at most $\eps|V_2|$ vertices in $V_2$ which are not $\eps$-typical. By the definition of regularity such vertices are incident to at most $\eps(1 + \eps)^2 n^2 p_0$ edges, as $d(V_1, V_2) \leq (1 + \eps) p_0$. Trivially, there at most $n$ vertices in $V_2$ that are incident to fewer than $(\delta/2)np_0$ edges in $E'$ and as all $\eps$-typical vertices have at most $(1 + \eps)^2 np_0$ neighbours in $V_1$, it follows that
  \[
    |\tilde{V}_2| \geq \frac{|E'| - \eps(1+\eps)^2 n^2p_0-n \cdot (\delta/2) n p_0}{(1 + \eps)^2 np_0} \geq \frac{(\delta/2 - 2\eps)n}{(1 + \eps)^2} \geq \frac{\delta}{4}n,
  \]
  by our choice of $\eps$.

  For an arbitrary vertex $v \in \tilde{V}_2$ we set $N_i := N(v, V_i)$, for $i \in \{1, 3\}$, and $N_1' := N_{E'}(v, V_1)$. Observe that $|N_1'| \geq (\delta /4) |N_1| \ge \eps|N_1|$ and hence, as $v$ is $\eps$-typical and $(N_1, N_3)$ forms an $(\eps, p)$-regular pair of density at least $(1 - \eps)^2p_0$, there is an edge between $N_1'$ and any subset of $N_3$ of size $\eps|N_3|$. Consequently, $E'$ expands to at least $|\tilde{V}_2| (1 - \eps)^3 np_0$ edges in $E(V_2, V_3)$. As $\tilde V_2$ is of size at least $\delta n/4 \geq \eps' n > \sqrt{\eps} n$ and $(V_2, V_3)$, $(V_2, V_4)$ are $(\eps, p)$-regular pairs, we can apply Proposition~\ref{prop:large-subsets-inheritance} to $(\tilde V_2, V_3)$ and $(\tilde V_2, V_4)$ to obtain that the tuple $(\tilde V_2, V_3, V_4)$ is $(\sqrt\eps, p)$-regular of density $(1 \pm (\eps + \eps/\alpha))p_0 \subseteq (1 \pm \eps')p_0$. Moreover, by Corollary~\ref{cor:regular-to-super-typical} applied to $(\tilde V_2, V_3, V_4)$ for $t = 3$, $\alpha$, $\eps'$, $\delta$, $\eps'$ as $\eta$, and $n$, we deduce that it is $(\delta, \eps')$-super-typical. Lastly, since $E'$ expands to at least $|\tilde{V}_2| (1 - \eps)^3 np_0$ edges in $E(V_2, V_3)$ and
  \[
    (1 - \eps)^3 |\tilde V_2| np_0 \geq (1 - \eps')^{5} |\tilde V_2| n (1 + \eps')^2 p_0 \geq (1 - \delta) |\tilde V_2| n (1 + \eps')^2 \alpha p
  \]
  we can apply Lemma~\ref{lem:one-step-expansion} with $\kappa = 1 - \delta$ to deduce that the edges in $E(\tilde V_2, V_3)$ to which $E'$ expands, in turn expand to at least $(1 - 10\delta)x_2$ edges in $(V_3, V_4)$, as claimed.

  Next, consider some $k \geq 3$ and assume the lemma holds for $k - 1$. We first give a brief overview of the proof. Let $\cK_k \subseteq K_k(V_1, \ldots, V_k)$ be an arbitrary subset of size $\delta x_k$. Similarly, as in the case $k=2$ we first define a set $\tilde V_k \subset V_k$ of vertices that are $\eps$-typical in all tuples $(V_i, \ldots, V_{i + k})$ for $i \in [k]$, and that belong to many copies of $K_k$ in $\cK_k$. For every vertex in $\tilde V_k$ we apply the induction hypothesis to its neighbourhoods in order to show expansion for every copy of $K_k$ lying on such a vertex. This allows us to show that we expand to almost all possible copies of $K_k$ in $(\tilde V_k, V_{k + 1}, \ldots, V_{2k - 1})$. Finally, similarly as above, a straightforward application of Lemma~\ref{lem:one-step-expansion} to the tuple $(\tilde V_k, V_{k + 1}, \ldots, V_{2k})$ concludes the proof. Let us now dive into the details.

  From now on we assume the lemma holds when applied for $k - 1$, $\alpha$, $\delta'$, and $(1 - \eps)^2 np_0$ as $n$. Since this happens with probability at least
  \[
    1 - e^{ -c_{k - 1} \cdot (1 - \eps)^2 n^2 p_0^2 p^{2(k- 3) + 1} } = 1 - e^{ -c_{k - 1} \cdot (1 - \eps)^2 \alpha^2 n^2 p^{2(k- 2) + 1} }
  \]
  it is sufficient to show that the induction step holds with probability at least $1 - e^{-\Omega(n^2 p^{2(k - 2) + 1})}$ with the hidden constant depending only on $\alpha, \delta$, and $k$, and then set $c$ to be sufficiently small with respect to that constant. In the remainder of the proof we further assume that $\GNp$ is such that the conclusion of Corollary~\ref{cor:regular-to-super-typical} for $\alpha$, $\eps'$, $\delta'$, $(1 - \eps')^2/(1 + \eps')^2$ as $\eta$, and $(1 + \eps)^2\alpha np$ as $n$, as well as for $k + 1$ as $t$, $\alpha$, $\eps'$, $\delta$, and $\eps'$ as $\eta$, holds. This happens with probability at least
  \[
    1 - e^{-c_1 \cdot \alpha^2 n^2 p^{2(k - 2) + 1}} - e^{-c_1 \cdot \eps'^2 n^2 p^{2(k - 2) + 1}} = 1 - e^{-\Omega(n^2 p^{2(k - 2) + 1})}.
  \]

  Let $\tilde V_k$ be the set of vertices $v \in V_k$ which are $\eps$-typical in $(V_i, \ldots, V_{i + k})$ for all $i \in [k]$, and which additionally belong to at least
  \[
    (\delta/2) n^{k-1} p_0^{\binom{k}{2}} = (\delta/2) x_{k - 1} p_0^{k - 1}
  \]
  copies of $K_k$ in $\cK_k$. In order to give a lower bound on the size of such a set, we first need an upper bound on the number of vertices $v \in V_k$ that belong to `many' canonical copies of $K_k$ in $(V_1, \ldots, V_k)$.

  Let $X \subseteq V_k$ be of size $\eps' |V_k|$ and assume that each $v \in X$ belongs to at least $(1 + \delta)^2 x_{k - 1} p_0^{k - 1}$ copies of $K_k$.
  By Proposition~\ref{prop:large-subsets-inheritance} $(V_1, \ldots, V_{k - 1}, X, V_{k + 1})$ forms a $(\sqrt\eps, p)$-regular $(k + 1)$-tuple of density $(1 \pm (\eps + \eps/\alpha))p_0 \subseteq (1 \pm \eps')p_0$. Moreover, by our choice of $\eps_0$ and $C$, it follows from Corollary~\ref{cor:regular-to-super-typical} that $(V_1, \ldots, V_{k - 1}, X, V_{k + 1})$ is $(\delta, \eps')$-super-typical. We thus know that
  \[
    |K_k(V_1, \ldots, V_{k - 1}, X)| \leq (1 + \delta) (1 + \eps')^{k^{2}} |X| n^{v(K_k)-1}p_0^{e(K_k)}.
  \]
  On the other hand, the assumption on the set $X$ implies that $|K_k(V_1, \ldots, V_{k - 1}, X)| \geq(1 + \delta)^2 |X| n^{v(K_{k - 1})} p_0^{e(K_k)}$, which is a contradiction as $(1 + \eps')^{k^2} < 1 + \delta$. Thus, such a set $X$ does not exist. Moreover, the vertices in $V_k$ which belong to at least $(1 + \delta)^2 x_{k - 1} p_0^{k - 1}$ copies of $K_k$ hence in total belong to at most
  \[
    (1 + \delta) (1 + \eps')^{k^{2}} \eps' n^{v(K_k)} p_0^{e(K_k)} \leq (1 + \delta)^2 \eps' x_k
  \]
  copies of $K_k$.

  Trivially, there are at most $n$ vertices which belong to fewer than $(\delta/2) x_{k - 1} p_0^{k - 1}$ copies of $K_k$ in $\cK_k$. As there are at most $\eps n$ vertices which are not $\eps$-typical in $(V_i, \ldots, V_{i + k})$ for each $i \in [k]$, we derive a lower bound on the size of $\tilde V_k$ as
  \[
    |\tilde V_k| \geq \frac{\delta x_k - (1 + \delta)^2 \eps' x_k- (\delta/2) x_k }{(1 + \delta)^2 x_{k - 1} p_0^{k - 1}} - k \eps n \geq \frac{\delta}{4}n - k \eps' n \geq \frac{\delta}{8} n.
  \]
  Let $v \in \tilde V_k$ be an arbitrary vertex and let $N_i := N(v, V_i)$ denote its neighbourhoods for all $i \in [2k - 1] \setminus \{k\}$. From the properties of typical vertices, our choice of $C_1$, and the bound on $n$ from the assumptions of the lemma, it follows that
  \[
    \min\{ C_1 p^{-k + 1}, C_1 p^{-k + 2}\log{N} \} \leq (1 - \eps)^{2}|V_i| \alpha p \leq | N_i| \leq (1 + \eps)^2|V_i|\alpha p.
  \]
  By Definition~\ref{def:typical-tuples}~$\ref{typical-tuple-reg-inheritance}$ the neighbourhoods of the vertex $v$ span $(\eps, p)$-regular tuples of densities $(1 \pm \eps)^2 p_0$. Let $\tilde n := (1 - \eps)^{2}n p_0$ and choose sets $\tilde N_i \subseteq N_i$ of size $|\tilde N_i| = \tilde n$ arbitrarily. By Proposition~\ref{prop:large-subsets-inheritance} the graph $G' := G[\tilde N_1, \ldots, \tilde N_{t - 1}, \tilde N_{t + 1}, \ldots, \tilde N_{2t - 1}]$ belongs to the class $\cG(P^{k - 1}_{2(k - 1)}, \tilde n, \alpha, \sqrt{\eps}, p)$. Additionally, we can apply Corollary~\ref{cor:regular-to-super-typical} for $\alpha$, $\eps'$, $\delta'$, $(1 - \eps')^{2}/(1 + \eps')^{2}$ as $\eta$, and $(1 + \eps)^{2}\alpha np$ as $n$, to obtain that $(\tilde N_i, \ldots, \tilde N_{i + k - 1})$ is $(\delta', \eps')$-super-typical for all $i \in [k]$. Therefore, $G' \in \cG_{\delta'}(P^{k - 1}_{2(k - 1)}, \tilde n, \alpha, \eps', p)$.

  Note that by taking subsets $\tilde N_i$ we may have destroyed some of the $(\delta/2) x_{k - 1} p_0^{k - 1}$ copies of $K_{k - 1}$ lying in the neighbourhood of $v$. However, note also that $|N_i\setminus \tilde N_i| \le 5\eps np_0$, with room to spare. Similarly as above, one can show that any set of size $\eps' np_0$ within $N_i$ belong to at most $(1 + \delta)^2 \eps' x_{k - 1} p_0^{k - 1}$ canonical copies of $K_{k - 1}$ in $(N_1, \ldots, N_{k - 1})$. Thus, at least
  \[
    (\delta/2) x_{k - 1} p_0^{k - 1} - k \cdot (1 + \delta)^2 \eps' x_{k - 1} p_0^{k - 1} \geq (\delta/4) x_{k - 1} p_0^{k - 1} \ge (\delta/4) \tilde n^{v(K_{k - 1})} p_0^{e(K_{k - 1})}
  \]
  canonical copies of $K_{k - 1}$ remain in $(\tilde N_1, \ldots, \tilde N_{k - 1})$.

  We now apply the induction hypothesis for $\delta'$ to obtain that starting from $(\delta/4) \tilde n^{v(K_{k - 1})} p_0^{e(K_{k - 1})}$ canonical copies of $K_{k - 1}$ in $(\tilde N_1, \ldots, \tilde N_{k - 1})$ we expand to at least
  \[
    (1 - 10\delta') \tilde n^{v(K_{k - 1})} p_0^{e(K_{k - 1})} \geq (1 - 10\delta') ((1 - \eps')^2 np_0)^{v(K_{k - 1})} p_0^{e(K_{k - 1})} \geq (1 - 12\delta') x_{k - 1} p_0^{k - 1}
  \]
  canonical copies of $K_{k - 1}$ in $(N_{k + 1}, \ldots, N_{2k - 1})$. This holds for every $v \in \tilde V_k$ and hence we expand to at least $(1 - 12\delta') |\tilde V_k| x_{k - 1} p_0^{k - 1}$ canonical copies of $K_k$ in $(\tilde V_k, V_{k + 1}, \ldots, V_{2k - 1})$.

  Since $|\tilde V_k| \geq (\delta/8) n > \eps' n$, we have that $(\tilde V_k, V_{k + 1}, \ldots, V_{2k - 1}, V_{2k})$ is an $(\sqrt\eps, p)$-regular $(k + 1)$-tuple of density $(1 \pm (\eps + \eps/\alpha))p_0 \subseteq (1 \pm \eps')p_0$ by Proposition~\ref{prop:large-subsets-inheritance}. Moreover, it is $(\delta, \eps')$-super-typical by Corollary~\ref{cor:regular-to-super-typical} applied with $k + 1$ as $t$, $\alpha$, $\eps'$, $\delta$, $\eps'$ as $\eta$, and $n$. Therefore, Lemma~\ref{lem:one-step-expansion} for $\kappa := 1 - \delta$, gives that starting from
  \begin{align*}
    (1 - 12\delta') |\tilde V_k| x_{k - 1} p_0^{k - 1} &\geq (1 - 12\delta')(1 - \eps')^{k^2} |\tilde V_k| n^{v(K_{k - 1})} ((1 + \eps')^2 \alpha p_0)^{e(K_k)} \\
    &\geq (1 - \delta) |\tilde V_k| n^{v(K_{k - 1})} ((1 + \eps')^2 \alpha p_0)^{e(K_k)}
  \end{align*}
  canonical copies of $K_k$ in $(\tilde V_k, V_{k + 1}, \ldots, V_{2k - 1})$, we expand to at least $(1 - 10\delta) x_k$ canonical copies of $K_k$ in $(V_{k + 1}, \ldots, V_{2k})$, as claimed.
\end{proof}

It turns out that having expansion from a small linear number of copies to almost all copies in some number of steps is enough in order to find a {\em single copy} of $K_k$ which expands to many copies in roughly $\log{N}$ more steps. The idea is a fairly simple one: if starting from a $\delta$-fraction of copies we expand to almost all copies in $k$ steps, then there has to exist a $(\delta/2)$-fraction of the copies which reaches roughly a half of all the copies in $k$ steps. By having $\delta$ small enough, Lemma~\ref{lem:main-expansion} shows that starting from the reached copies we expand again to almost all the copies in another $k$ steps. In conclusion, we found a $(\delta/2)$-fraction of the initial copies that in $2k$ steps expand to almost all possible copies. Repeating this procedure for roughly $\log N$ times yields the result.
\begin{proof}[Proof of Lemma~\ref{lem:clique-expansion-lemma}]
  Write $p_0 = \alpha p$. For given $\alpha, \delta, k$, let $\eps_0 = \eps_{0_{\ref{lem:main-expansion}}}(\alpha, \delta, k)$, $C = C_{\ref{lem:main-expansion}}(\alpha, \delta, k)$, and $c = c_{\ref{lem:main-expansion}}(\alpha, \delta, k)$. We assume that $\GNp$ is such that the conclusion of Lemma~\ref{lem:main-expansion} holds for $\alpha$, $\delta$, and $k$. This happens with probability at least $1 - e^{-c \cdot n^2 p^{2(k - 2) + 1}}$.

  Let $\cK$ be an arbitrary subset of $K_k(V_1, \ldots, V_k)$ of size $\delta n^{v(K_k)} p_0^{e(K_k)}$. By Lemma~\ref{lem:main-expansion}, $\cK$ expands to at least $(1 - 10k\delta) n^{v(K_k)} p_0^{e(K_k)}$ canonical copies of $K_k$ in $(V_{k + 1}, \ldots, V_{2k})$. The following claim captures the main property we need in order to show the assertion of the lemma.
  \begin{claim}\label{cl:halving}
    Let $1 \leq i \leq \floor{\ell/k} - 2$ be an integer and $A \subseteq \cK$ such that $A$ expands to at least $2 \delta n^{v(K_k)} p_0^{e(K_k)}$ canonical copies of $K_k$ in $(V_{ik + 1}, \ldots, V_{(i + 1)k})$. Then there exists a subset $A' \subseteq A$ such that $|A'| \leq \ceil{|A|/2}$ and $A'$ expands to at least $(1 - 10k\delta)n^{v(K_k)}p_0^{e(K_k)}$ canonical copies of $K_k$ in $(V_{(i + 1)k + 1}, \ldots, V_{(i + 2)k })$.
  \end{claim}
  Before showing the claim let us first complete the proof of the lemma. Take $m = \floor{\ell/k} - 2$ and apply Claim~\ref{cl:halving} repeatedly $m$ times. This shows that there exist a set $\cK' \subseteq \cK$ such that $|\cK'| \leq \ceil{|\cK|/2^{m}}$ and $\cK'$ expands to at least $(1 - 10k\delta) n^{v(K_k)} p_0^{e(K_k)}$ canonical copies of $K_k$ in $(V_{(m + 1)k + 1}, \ldots, V_{(m + 2)k})$. Since $|\cK| \leq N^{k}$ and $m \geq 2k \log N$, it follows that $|\cK|/2^{m} \leq 1$. Hence, there exists a $K \in \cK$ which expands to at least $(1 - 10k\delta)n^{v(K_k)}p_0^{e(K_k)}$ canonical copies of $K_k$ in $(V_{(m + 1)k + 1}, \ldots, V_{(m + 2)k})$. Finally, by repeated application of Lemma~\ref{lem:one-step-expansion} for at most $k$ times with $\kappa := 1 - 10k\delta$ we further get that $K$ expands to a set $\cK' \subseteq K_k(V_{\ell - k + 1}, \ldots, V_{\ell})$ of size $|\cK'| \geq (1 - 20 k \delta) n^{v(K_k)} p_0^{e(K_k)}$.

  \begin{proof}[Proof of Claim~\ref{cl:halving}]
    Take an arbitrary partition $A = S \cup T$ such that $|S|, |T| \leq \ceil{|A|/2}$. As $A$ expands to at least $(1 - 10k\delta) n^{v(K_k)} p_0^{e(K_k)}$ canonical copies of $K_k$ in $(V_{ik + 1}, \ldots, V_{(i + 1)k})$, it follows that either $S$ or $T$ has to expand to at least $(1/2 - 5k\delta)n^{v(K_k)}p_0^{e(K_k)}$ canonical copies of $K_k$ in $(V_{ik + 1}, \ldots, V_{(i + 1)k})$. Since $(1/2 - 5k\delta)\ge \delta$ we can apply Lemma~\ref{lem:main-expansion} to this set of canonical copies of $K_k$ to deduce that it expands to at least $(1 - 10k\delta) n^{v(K_k)} p_0^{e(K_k)}$ canonical copies of $K_k$ in $(V_{(i + 1)k + 1}, \ldots, V_{(i + 2)k })$.
  \end{proof}
  This completes the proof of the lemma.
\end{proof}

\subsection{Proof of Theorem~\ref{thm:main-theorem}}
With all the tools at hand we can now prove our main theorem, which we restate here for the convenience of the reader.
\MainTheorem*
\begin{proof}
  Let $G$ be a spanning subgraph of $\GNp$ with $\delta(G) \geq (\frac{k}{k + 1} + \alpha)Np$. In the rest of the proof we define some constants and use the following relation between them
  \[
    0 < \eps'' \ll \eps_0'' \ll \eps' \ll \eps_0' \ll \tilde \eps \ll \tilde \eps_0 \ll \xi \ll \eps < 1,
  \]
  where by $a \ll b$ we mean that $a$ is chosen to be sufficiently smaller than $b$. Now we make this precise. Take $\delta = (20k)^{-6}$, $\eta = 1$, let $d = d_{\ref{cor:nice-regularity}}(\alpha)$, $\xi = \min\{\eps/4, d/(d + 1)\}$, $\tilde\eps_0 = \eps_{0_{\ref{lem:clique-expansion-lemma}}}(d, \delta, k)$, and $\tilde\eps = \tilde\eps_0/2$. Let $\eps_0' = \eps_{0_{\ref{cor:regular-to-super-typical}}}(d, \tilde\eps, \delta, \eta, k + 1)$, $\eps' = \eps_0' \xi/2$, $\eps_0'' = \eps_{0_{\ref{lem:regularity-partitioning}}}(\eps', d)$, and $\eps'' = \eps_0''/2$. Define $c_1 = c_{\ref{lem:regularity-partitioning}}(\eps', d)$, $c_2 = c_{\ref{cor:regular-to-super-typical}}(d, \tilde\eps, \delta, \eta, k + 1)$, and $c_3 = c_{\ref{lem:clique-expansion-lemma}}(d, \delta, k)$. Lastly, set $n_0 = n_{0_{\ref{thm:KSS}}}(k)$, $M = M_{\ref{cor:nice-regularity}}(\eps'', n_0)$, and $C = C_{\ref{lem:clique-expansion-lemma}}(d, \delta, k)$. Throughout the proof we assume that $\GNp$ is such that:
  \begin{itemize}
    \item the conclusion of Corollary~\ref{cor:nice-regularity} holds for $k/(k + 1)$ as $\mu$, $\alpha$ as $\nu$, $\eps''$ as $\eps$, and $n_0$ as $m$ (which happens with probability at least $1 - e^{-\Omega(N^2p)}$),
    \item the conclusion of Lemma~\ref{lem:regularity-partitioning} holds for $\eps'$, $d$, and some $q \gg p^{-1}\log N$ (which happens with probability at least $1 - e^{-c_1 \cdot q^2p}$),
    \item the conclusion of Corollary~\ref{cor:regular-to-super-typical} holds for $k + 1$ as $t$, $d$ as $\alpha$, $\tilde\eps$ as $\eps'$, $\delta$, $\eta$, and $\xi q$ as $n$ (which happens with probability at least $1 - e^{-c_2 \cdot (\xi q)^2 p^{2(k - 2) + 1}}$), and
    \item the conclusion of Lemma~\ref{lem:clique-expansion-lemma} holds for $k$, $d$ as $\alpha$, $\delta$, and $\xi q$ as $n$ (which happens with probability at least $1 - e^{-c_3 \cdot (\xi q)^2 p^{2(k - 2)} + 1}$).
  \end{itemize}
  For $q \geq (1 - \eps'')N/(M \cdot 3k^2 \log N)$ we have
  \[
    (\xi q)^2 p^{2(k - 2) + 1} \geq \xi^2 \left( \frac{(1 - \eps'') N}{M \cdot 3k^2 \log N} \right)^2 \cdot C^{2(k - 2) + 1} \left( \frac{\log N}{N} \right)^{2 - 3/k}.
  \]
  Since all constants depend only on $\eps$, $\alpha$, and $k$, by choosing $c$ to be small enough w.r.t.\ those the success probability is at least $1 - e^{-c \cdot (N/\log N)^{3/k}}$, as required. Conditioning on the above, the rest of the proof is fully deterministic.

  We apply Corollary~\ref{cor:nice-regularity} to $G$ with $\eps''$ as $\eps$, $k/(k + 1)$ as $\mu$, $\alpha$ as $\nu$, and $n_0$ as $m$ to obtain a partition $V(G) = V_0 \cup V_1 \cup \ldots \cup V_{t_0}$, with $n_0 \leq t_0 \leq M$, such that $|V_0| \leq \eps'' N$, $|V_1| = \ldots = |V_{t_0}| = n'' \in [(1 - \eps'')N/t_0, N/t_0]$, and for every $i \in [t_0]$ there are at least $kt_0/(k + 1)$ indices $j \in [t_0] \setminus \{i\}$ such that $G[V_i, V_j]$ is an $(\eps'', p)$-regular pair of density $dp$. As $t_0 \geq n_0$, Theorem~\ref{thm:KSS} yields that $G$ contains a subgraph $G'' \in \cG(C^{k}_{t_0}, n'', d, \eps'', p)$ on the partition classes $V_1, \ldots, V_{t_0}$, where w.l.o.g.\ $V_i$ represents the $i$-th vertex of the cycle. Furthermore, since $\eps'' < \eps_{0_{\ref{lem:regularity-partitioning}}}(\eps', d)$, we may apply Lemma~\ref{lem:regularity-partitioning} for $\eps'$ and $n' = \floor{n''/r}$ as $q$, where $r = 3k^2 \log N$, to $(V_1, \ldots, V_{t_0})$ in order to obtain a new partition
  \[
    V(G) = V_0 \cup \bigcup_{i = 1}^{t_0} \bigcup_{j = 0}^{r} V_i^j,
  \]
  such that $|V_i^0| \leq n'$ and $|V_i^1| = \ldots = |V_i^r| = n'$, for all $i \in [t_0]$. For convenience, we rename the partition classes as follows: $\tilde V_0 = V_0 \cup V_1^0 \cup \ldots V_{t_0}^0$ and $\tilde V_{t_0 \cdot (j - 1) + i} = V_i^j$, for $i \in [t_0]$ and $j \in [r + 1]$. Furthermore, we set $t = t_0 r$ and identify $\tilde V_{t + i}$ with $\tilde V_i$, for every $i \in [t]$.

  Observe that, as $G'' \in \cG(C^{k}_{t_0}, n'', d, \eps'', p)$, we have that for all $i \in [t_0]$ and $j \in [k]$, $(V_i, V_{i + j})$ (where we again identify $V_{t_0 + i}$ with $V_{i}$) is an $(\eps'', p)$-regular pair of density $dp$. In turn this implies that all $(\tilde V_{t_0 \cdot (r_1 - 1) + i}, \tilde V_{t_0 \cdot (r_2 - 1) + i + j})$, with $r_1, r_2 \in [r + 1]$, form $(\eps', p)$-regular pairs of density $(1 \pm \eps')dp$. Thus we have that $G$ contains a subgraph $G' \in \cG(C^{k}_{t}, n', d, \eps', p)$ on the partition classes $\tilde{V}_1, \ldots, \tilde{V}_t$ (in this order).

  We set $\tilde n = \xi n'$ and $x = \tilde n^{v(K_k)} (dp)^{e(K_k)}$. By Proposition \ref{prop:large-subsets-inheritance} for any collection of subsets $X_i \subseteq \tilde V_i$ of size $|X_i| = \tilde n \geq \eps' n'$, the graph $\tilde G$ induced by the sets $X_1, \ldots, X_t$ belongs to the class $\cG(P^{k}_{t}, \tilde n, d, \tilde \eps, p)$, since $1 \pm (\eps' + \eps'/d) \subseteq 1 \pm \tilde \eps$. This observation, together with the fact that $\eps'/\xi < \eps_{0_{\ref{cor:regular-to-super-typical}}}(d, \tilde\eps, \delta, \eta, k + 1)$ and Corollary~\ref{cor:regular-to-super-typical}, implies that the following holds:
  \begin{enumerate}[label=(P\arabic*), font=\itshape\bfseries, leftmargin=2.8em]
    \item\label{main-proof-P1} For all $1 \leq \ell \leq N$, every fixed $\ell$-partite subgraph $\tilde G$ of $G$ induced by the parts $(X_1, \ldots, X_\ell)$ of pairwise disjoint subsets $X_i \subseteq \tilde V_i$ of size $|X_i| = \tilde n$ belongs to the class $\cG_\delta(P^{k}_{\ell}, \tilde n, d, \tilde \eps, p)$.
  \end{enumerate}

  Property~\ref{main-proof-P1}, Lemma \ref{lem:clique-expansion-lemma}, and our choice of $\tilde \eps < \eps_{0_{\ref{lem:clique-expansion-lemma}}}(d, \delta, k)$ furthermore imply that another important property holds:
  \begin{enumerate}[resume, label=(P\arabic*), font=\itshape\bfseries, leftmargin=2.8em]
    \item\label{main-proof-P2} For all $1 \leq \ell \leq N$, every fixed tuple $\mathbf{X} = (X_1, \ldots, X_\ell)$ such that $X_i \subseteq \tilde V_i$ are pairwise disjoint subsets of size $|X_i| = \tilde n$, and every $\cK \subseteq K_k(X_1, \ldots, X_k)$ of size $|\cK| \geq \delta x$, there exist $K \in \cK$ and $\cK' \subseteq K_k(X_{\ell - k + 1}, \ldots, X_\ell)$ of size $|\cK'| \ge (1 - 20k\delta)x$ such that $K$ expands to $\cK'$ over $\mathbf{X}$.
  \end{enumerate}

  Note that here we do allow $\ell > t$, which in particular means that we may apply properties \ref{main-proof-P1} and \ref{main-proof-P2} to tuples of the form $(X_1, \ldots, X_t, X_1', \ldots, X_k')$ with $X_i, X_i' \subseteq V_i$ and $X_i \cap X_i' = \varnothing$, for $i \in [k]$.
  Moreover, we also allow the sets $X_i$ to be chosen `backwards', that is a tuple of the form $(X_k, \ldots, X_1, X_t', \ldots, X_1')$, by relabelling the partition classes $\tilde V_1, \ldots, \tilde V_t$ of $C_t^k$ in the opposite order if necessary.

  With this at hand, we may proceed with the main argument. We first choose tuples $(S_1, \ldots, S_t)$ and $(T_{1}^{1}, \ldots, T_{t}^{1})$ where all $S_i, T_i^{1} \subseteq \tilde V_i$ are of size $\tilde n$, and $S_i \cap T_i^{1} = \varnothing$, for all $i \in [t]$. Combining the properties of super-typical tuples with \ref{main-proof-P2} from above it is not hard to see that there exists a single canonical copy of $K_k$ in $(T_1^1, \ldots, T_k^1)$ that expands to almost all canonical copies of $K_k$ in $(T_{t - k + 1}^{1}, \ldots, T_{t}^{1})$ and at the same time expands `backwards' to almost all canonical copies of $K_k$ in $(S_{1}, \ldots, S_{k})$ through $(T_{k}^{1}, \ldots, T_{1}^{1}, S_t, \ldots, S_1)$. Indeed, observe that \ref{main-proof-P2} applied for $(T_1^1, \ldots, T_t^1)$ implies that all but at most $\delta x$ copies of $K_k$ in $K_k(T^1_1, \ldots, T^1_k)$ are such that they expand to at least $(1 - 20k\delta)x$ copies of $K_k$ in $K_k(T^1_{t-k+1}, \ldots, T^1_t)$. Similarly, applying \ref{main-proof-P2} this time to $(T_k^1, \ldots, T_1^1, S_t, \ldots, S_1)$, implies that all but at most $\delta x$ copies of $K_k$ in $K_k(T^1_1, \ldots, T^1_k)$ are such that they expand to at least $(1 - 20k\delta)x$ copies of $K_k$ in $K_k(S_1, \ldots, S_k)$. As $K_k(T^1_1, \ldots, T^1_k)$ contains more than $2\delta x$ copies of $K_k$ we know that it contains at least one such copy that expands well to both sides.

  Fix such a copy $K^*$ and denote by $\cK^{1}$ and $\cK^{*}$ the sets of copies in $K_k(T_{t - k + 1}^{1}, \ldots, T_{t}^{1})$ and $K_k(S_1, \ldots, S_k)$ to which $K^*$ expands to, respectively. The set $\cK^{*}$ together with the tuple $(S_1, \ldots, S_t)$ are set aside and used in the end to close the cycle. On the other hand, the set $\cK^{1}$ is used to inductively build a long path covering almost all vertices of $V(G) \setminus (S_1 \cup \ldots \cup S_t)$. More precisely, as long as the sets $\tilde V_i$ contain at least $2 \tilde n$ vertices which do not belong to the previously built path, by making use of the property \ref{main-proof-P2}, we extend the path by $t$ vertices.

  We prove by induction that for any $s \in \{ 1, \ldots, \floor{(1 - 2 \xi) n'} \}$ the following holds:
  \begin{enumerate}[(i), font=\itshape]
    \item\label{main-path} there exist sets $P^s_i \subseteq \tilde V_i \setminus S_i$ of size $|P^s_i| = s - 1$, for all $ i \in [t]$,
    \item\label{main-leftover} there exists $\mathbf{T}^{s} = (T^s_1, \ldots, T^s_t)$ such that $T^s_i \subseteq \tilde V_i \setminus (S_i \cup P^s_i)$ and $|T^s_i| = \tilde n$, for all $i \in [t]$,
    \item\label{main-expansion} there exists a subset of canonical copies $\cK^{s} \subseteq K_k(T^{s}_{t - k + 1}, \ldots, T^s_t)$ of size $|\cK^s| \ge (1 - 20k\delta) x$ such that for all $K' \in \cK^s$ there exists a $k$-path $P$ which starts in $K^*$, ends in $K'$, $V(P) \cap \tilde V_i \subseteq P^s_i \cup T^s_i$, and $|V(P) \cap \tilde V_i| = s$ for all $i \in [t]$.
  \end{enumerate}

  The sets $P_i^{s}$ represent vertices of the partition classes $\tilde V_i$ used by the current path and the tuples $\mathbf{T}^s$ in $\ref{main-leftover}$ are used to show expansion of the current path starting at the previously fixed copy $K^*$ to many canonical copies of $K_k$ in $(T_{t - k + 1}^{s}, \ldots, T_{t}^{s})$.

  The base of the induction, i.e.\ $s = 1$, holds directly by the choice of $K^*$, $\cK^{1}$, and by setting $P^1_i = \varnothing$, for all $i \in [t]$. Assume now that the hypothesis holds for some $1 \leq s < \floor{(1 - 2\xi)n'}$ and let us show that it holds for $s + 1$.

  Let $T_i^{s + 1} \subseteq \tilde V_i \setminus (P^{s}_i \cup T^{s}_i)$ be an arbitrary subset of size $|T^{s + 1}_i| = \tilde n$, for all $i \in [t]$. As
  \[
    |\tilde V_i \setminus (P^{s}_i \cup T^{s}_i)| \ge n' - ((1 - 2\xi)n' + \xi n') \ge \xi n' = \tilde n,
  \]
  such a subset indeed exists. Since $1 - 20k\delta \geq \delta$, by the induction hypothesis and property \ref{main-proof-P2} applied to the $(k + t)$-tuple $(T^{s}_{t - k + 1}, \ldots, T^{s}_t, T^{s + 1}_{1}, \ldots, T^{s + 1}_t)$, we have that there exist $K' \in \cK^{s}$ and $\cK^{s + 1} \subseteq K_k(T^{s + 1}_{t - k + 1}, \ldots, T^{s + 1}_t)$ of size $|\cK^{s + 1}| = (1 - 20k\delta) x$ such that $K'$ expands to $\cK^{s + 1}$ over $(T^{s}_{t - k + 1}, \ldots, T^{s}_t, T^{s + 1}_{1}, \ldots, T_t^{s + 1})$. Again making use of the induction hypothesis, there exists a $k$-path $P'$ which starts in $K^*$, ends in $K'$, $V(P') \cap T^{s + 1}_i = \varnothing$, and $|V(P') \cap \tilde V_i| = s$. Finally, we set $P^{s + 1}_i = V(P') \cap \tilde V_i$. One easily checks that properties $\ref{main-path}$--$\ref{main-expansion}$ are now satisfied.

  Note that when $s = \floor{(1 - 2 \xi) n'}$ property $\ref{main-expansion}$ tells us that there exist `many' $k$-paths which start in $K^*$ and are of length
  \begin{align*}
    \floor{(1 - 2 \xi) n'} t &\geq N - \eps'' N - n' t_0 - 2 \xi n' t \geq N - \eps'' N - n' t_0 - 2\xi n' r t_0 \\
    &\geq (1 - \eps'')N - 3\xi n' r t_0 \geq (1 - \eps'' - 3\xi)N \geq (1 - \eps)N.
  \end{align*}
  It remains to show that at least one of these paths can be closed into a $k$-cycle. Indeed, this is possible due to our initial choice of $K^*$. Let $(S'_1, \ldots, S'_t)$ be a tuple such that $S'_i \subseteq \tilde V_i \setminus (P_i^{s} \cup T_i^{s} \cup S_i)$ and $|S'_i| = \tilde n $ for all $i \in [t]$, where $s = \floor{(1 - 2\xi)n'}$. By applying property \ref{main-proof-P2} to the $(2k + t)$-tuple
  \[
    \mathbf{C} = (T_{t - k + 1}^{s}, \ldots, T_{t}^{s}, S'_{1}, \ldots, S'_{t}, S_{1}, \ldots, S_k)
  \]
  and the set $\cK^s$, we obtain a canonical copy $K \in \cK^s$ and a set $\tilde \cK \subseteq K_k(S_{1}, \ldots, S_k)$ of size $\tilde \cK \ge (1 - 10k \delta) x$ such that $K$ expands to $\tilde \cK$ over $\mathbf{C}$. Recall, as $\cK^{*} \subseteq K_k(S_1, \ldots, S_k)$ is of size at least $(1 - 20k\delta)x$ and $K^*$ expands to every copy of $K_k$ in $\cK^{*}$, we can connect $K^*$ via the sets $S_t, \ldots, S_1$ by a $k$-path to some copy of $K_k$ in $\tilde \cK \cap \cK^{*}$, which in turn closes the desired $k$-cycle.
\end{proof}

%% file: acknowledgements.tex
\paragraph*{Acknowledgements.} We would like to thank the anonymous reviewers for exceptionally careful reading of our paper and many helpful comments which led to the improved presentation.